\documentclass{article} 
\usepackage{hyperref}
\usepackage{amssymb}
\usepackage{amsmath,amsthm, latexsym, wasysym}
\usepackage{enumerate}
\usepackage{amsfonts}
\usepackage[numbers]{natbib} 
\newtheorem{theorem}{Theorem}[section]
\newtheorem{lemma}[theorem]{Lemma}
\newtheorem{corollary}[theorem]{Corollary}
\newtheorem{proposition}[theorem]{Proposition}
\newtheorem{remark}[theorem]{Remark}
\theoremstyle{definition}
\newtheorem{definition}[theorem]{Definition}

\numberwithin{equation}{section}
\usepackage[T1]{fontenc}
\usepackage{mathrsfs}
\usepackage{setspace}
\usepackage{graphicx,subfigure}
\usepackage{xcolor}
\usepackage{epstopdf}
\usepackage[utf8]{inputenc}
\usepackage{multicol}
\usepackage{amsmath}
\usepackage{amsthm}
\usepackage{upgreek}
\usepackage{amsfonts}
\usepackage{contour}
\usepackage{float}
\usepackage{breqn}
\usepackage{textcomp}
\usepackage{gensymb}
\usepackage[left=1.3in,right=1in,top=1in,bottom=1in]{geometry}
\usepackage{hyperref}
\usepackage{lipsum}
\usepackage{apptools}
\usepackage{natbib,hyperref}
\newcommand{\remove}[1]{}
\newcommand{\ra} {\rightarrow}
\newcommand{\RR} {\mathbb{R}}
\newcommand{\DD} {\displaystyle}
\newcommand{\la} {\lambda}
\newcommand{\Om}{\Omega}
\theoremstyle{remark}
\newcommand{\R}{\mathbb R^N\times\mathbb R^N}
\newcommand{\A}{\int_{\mathbb R^N}\int_{\mathbb R^N}\frac{1}{p(x,y)}\frac{|u(x)-u(y)|^{p(x,y)}}{|x-y|^{N+s(x,y)p(x,y)}}dxdy}
\newcommand{\B}{\DD\int_{\Omega}V(x)\frac{|u(x)|^{\overline{p}(x)}}{\overline{p}(x)}dx}

\newcommand{\al}{\alpha}
\newcommand{\ba}{\beta}

\newcommand{\p}{\overline{p}(x)}
\newcommand{\F}{\int_{\Om}\int_{\Om}\frac{F(x,u(x))F(y,u(y))}{|x-y|^{\mu(x,y)}}dxdy}
\newcommand{\f}{\int_{\Om}\int_{\Om}\left(\frac{F(y,u(y))}{|x-y|^{\mu(x,y)}}dy\right)f(x,u(x))u(x)dx}
\newcommand{\rr}{(x,t)\in\Om\times \mathbb R}

\newcommand{\br}{\in\mathbb R}
\newcommand{\gr}{\frac{\theta p^+}{2}}
\newcommand{\noi}{\noindent}
\newcommand {\n}{\nonumber\\}
\newcommand{\gt}{n\to\infty}
\newcommand{\ol}{\overline}
\newcommand{\e}{\epsilon}
\title{On a class of Kirchhoff-Choquard equations involving variable-order fractional $p(\cdot)-$ Laplacian and without Ambrosetti-Rabinowitz type condition}
\author{{Reshmi Biswas\thanks{ Email: b.reshmi@iitg.ac.in} } and {Sweta Tiwari\thanks{Email: swetatiwari@iitg.ac.in}}
	\\{\small Department of Mathematics}, {\small Indian Institute of Technology Guwahati,}\\{\small Guwahati, Assam 781039, India. }}
\date{}
\begin{document}

\maketitle

\begin{abstract}
\noi	In this article we study the existence of weak solution, existence  of ground state solution
using Nehari manifold and existence of infinitely many solutions using Fountain theorem and Dual fountain theorem for a class of doubly nonlocal Kirchhoff-Choquard type equations
involving the variable-order fractional $p(\cdot)-$ Laplacian operator. Here the nonlinearity does not satisfy the well known  Ambrosetti-Rabinowitz type condition.
\end{abstract}

\noi {\bf Keywords}:   Kirchhoff-Choquard equation; Variable order fractional $p(\cdot)$- Laplacian;\\ Fountain theorem; Dual fountain theorem; Nehari manifold; Ambrosetti-Rabinowitz type condition\\

\noi {\bf Mathematics subject classification:}	35J60; 35R11; 35A15; 46E35
	\section{Introduction }
	\noi The purpose of this article is to study the following Kirchhoff-Choquard type problem:
{\begin{equation}\label{mainprob}
	\;\;\;	\left.\begin{array}{rl}
	& m\bigg(\DD\A+\B\bigg)\n	&~~~~~~~~~~~~~~\left[(-\Delta)_{p(\cdot)}^{s(\cdot)}u+V(x)|u|^{\p-2}u\right] =\left(\DD\int_\Om\frac{F(y,u(y))}{|x-y|^{\mu(x,y)}}dy\right)f(x,u),\hspace{4mm}
	x\in \Om,\\\\
	&\hspace{6.7cm}u=0,~~~~~~~~~~~~~~~~~~~~~~~~~~~~~~~~x\in \RR^N\setminus \Om,
	\end{array}
	\right\}
	\end{equation}}
where $m:\RR_0^+\to\RR_0^+,$  $V:\Om\to\RR_0^+,$ $\mu:\R\to(0,N),$ $f:\Om\times\RR\to\RR$ $p:\R\to(1,\infty)$ and $s:\R\to(0,1)$ are  continuous functions, $\p:=p(x,x),$ $\Om\subset\RR^N$ is  a smooth bounded   domain and $F$ is the primitive of $f.$ The nonlocal operator  $(-\Delta)_{p(\cdot)}^{s(\cdot)}$ is defined  as
{\begin{equation}\label{operator}
(-\Delta)_{p(\cdot)}^{s(\cdot)} u(x):=  P.V.
\int_{\RR^N}\frac{\mid
	u(x)-u(y)\mid^{p(x,y)-2}(u(x)-u(y))}{\mid
	x-y\mid^{N+s(x,y)p(x,y)}}dy, ~~x \in \RR^N,
\end{equation}}
where P.V. denotes the Cauchy's principal value. Note that such kind of operators are non-homogeneous in nature. If $p(x,y)=p,~s(x,y)=s$ are constants then \eqref{operator} is reduced to nonlocal fractional $p-$Laplacian. 
The fractional Sobolev spaces and the corresponding nonlocal equations involving nonlocal operator have major applications
to various nonlinear problems, including phase transitions, thin obstacle problem, crystal dislocation, soft thin films, minimal surfaces, material science, etc.
 (see for e.g., \cite{caff2} and the references there in for more details).
We also refer  to the  monograph \cite{bisci} and \cite{hitchhiker} for the detailed study of the nonlocal fractional $p-$Laplacian with $1<p<\infty,$
its properties and problems involving it.\\	
 Motivated by the application of variable exponent Lebesgue and Sobolev spaces (see \cite{diening,radulescu2,radulescu1}) in the study of
the flow of non-Newtonian fluid (for example electrorheological fluids) or the image restoration ( see for e.g., \cite{ant,chen}),
Kaufmann et al. (\cite{kaufmann})   introduced 
fractional Sobolev spaces with variable exponents. Then R\u{a}dulescu, Bahrouni, Ho, Kim (\cite{bahrouni-jmaa,bahrouni-dcds,ky-ho})  studied the extensive properties of such spaces and associated problems involving the operator \eqref{operator} when $s(\cdot,\cdot)=s,$ constant.	For $p(\cdot,\cdot)=2,$ Zhang et al. (\cite{new})  discussed  existence results for  problems involving  variable-order fractional Laplacian operator $(-\Delta)^{s(\cdot)}$.
Recently,
Biswas and Tiwari (\cite{rs})  studied   problems  involving the nonlocal operator \eqref{operator} in  fractional Sobolev spaces with variable order and variable exponents, which involve not only the  variable exponents
but also  variable order.\\	
The Choquard type of nonlinearity 
in \eqref{mainprob} is motivated 
by the work of Pekar \cite{pekar}, 
studying the following nonlinear Schr\"{o}dinger-Newton equation.
{\begin{align}\label{sn}
-\Delta u + V(x)u = ({\mathcal{K}}_\mu * u^2)u +\la f(x, u),
\end{align}}
where $\mathcal{K}_\mu$ denotes is the  Riesz potential.
This type of nonlinearity describes the self gravitational
collapse of a quantum mechanical wave function (see \cite{penrose})
and also plays a key role in the Bose–Einstein
condensation (see \cite{bose}).
For $V(x)=1,\la=0$, the equations of type \eqref{sn} have extensively been studied in  (\cite {lieb,lions,moroz-main,moroz}). We cite \cite{radulescu6,chqd2} for some recent works related to critical Choquard type problems involving local Laplacian.
In the fractional Laplacian set up, Wu (\cite{wu}) investigated existence
and stability of solutions for the equations
{\begin{align}\label{sn1}
(-\Delta)^s u + \omega u = ({\mathcal{K}}_\mu *|u|^q)|u|^{q-2}u + \la f(x,u) ~~\text{in} ~\RR^N ,\end{align}}
where $q = 2, ~\la = 0$ and $\mu\in (N -2s, N).$  
In the
critical case, i.e. $q = 2_{\mu,s}^*:=(2N -{2{\mu}/2})/(N - 2s),$ Mukherjee and Sreenadh (\cite{tuhina}) obtained existence and multiplicity results for
solutions of \eqref{sn1}  in a smooth bounded domain for $w=0$ and $f(x, u) = u.$

Recently Gao et al. (\cite{gao}) investigated the existence of
ground state solution of Pohozaev-type  for the following  problem:
{\begin{align}\label{sn1.0}
(-\Delta)^s u + V(x) u = ({\mathcal{K}}_\mu *F(u))F'(u) ~~\text{in} ~\RR^N ,\end{align} }
where $V\in C^1(\RR^N,[0,\infty))$ and  $F$  satisfies  general Berestycki–Lions-type assumptions.\\
Very recently, Alves et al. (\cite{alves-choquard}) introduced a new kernel 
$A(x, y) :=
\frac{1}{|x - y|^{\mu(x,y)}}$  for $ x, y \in \RR^N$
and then using the properties of that kernel the authors established   Hardy-Littlewood-Sobolev-type
inequality (\cite[Proposition 2.4]{alves-choquard}) for variable exponents.
Analogous to this, in the  nonlocal  setup involving  variable order and variable exponents, Biswas and Tiwari (\cite{rs}) established  a Hardy-Littlewood-Sobolev-type
inequality result   with some  appropriate assumptions and studied the Choquard problem of type \eqref{mainprob} for $m(\cdot)=1, V\equiv0$ 
 and $f\in C(\Om\times\RR,\RR)$ satisfying Ambrosetti-Rabinowitz type condition.\\	
 The  study  of  Kirchhoff-type problems  arise  in  various  models  of  physical  and  biological  systems and hence have received  more attentions in recent years. Precisely, 
Kirchhoff established a model given by the following equation:
{$$\rho\frac{\partial^2 u}{\partial^2t}-\left(\frac{p_0}{h}+\frac{E}{2L}\int_0^L\left|\frac{\partial u}{\partial t}\right|^2dx\right)\frac{\partial^2 u}{\partial^2t}=0,$$ } which extends the classical D’Alembert wave equation by taking into account the effects of the changes in the length of
the strings during the vibrations,
where the constants $\rho, p_0, h, E, L$ 
represent physical parameters of the string.
Subsequently,  using the Nehari manifold and the concentration
compactness principle  in \cite{lu}, L\"{u} studied the Kirchhoff-Choquard-type equation
{\begin{align}\label{sn2}
	\Big(-a+b \int_{\RR^3}
	|\nabla u|^2 dx\Big)\Delta u + V_\la(x)u = ({\mathcal{K}}_\mu * |u|^q)|u|^{q-2}u ~~\text{in}~ \RR^3,\end{align}} where $a \in \RR^+, b\in \RR_0^+,$
$ V_\la(x) = 1 + \la g(x)$, $\la>0$ and $g$ is a continuous
potential function, $q \in (2, 6 -\mu).$ 
Fiscella and Valdinoci in \cite{valdi} first
proposed a stationary Kirchhoff model  involving fractional Laplacian   
 by considering the nonlocal aspect of the tension arising from
nonlocal measurements of the fractional length of the string,  (\cite{valdi}, Appendix A).
Recently, Pucci et al. (\cite{pucci-choquard}) extensively studied the  existence and asymptotic behavior of entire solutions of the following
superlinear Kirchhoff-Schr\"{o}dinger-Choquard equation involving fractional $p-$Laplacian: 
{\begin{align}\label{sck}
M(\|u\|^p_s)[(-\Delta)_s^p u + V(x)|u|^{p-2}u] = \la f(x, u) + ({\mathcal{K}}_\mu* |u|^{p_{\mu,s}^*} )|u|^{p_{\mu,s}^*-2}u ~~\text{ in}~ \RR^N ,\end{align}}
where $M : \RR_0^+ \to\RR_0^+$ is degenerate-type Kirchhoff function, $V : \RR^N \to\RR^+$ is a scalar potential, $p_{\mu,s}^* = (pN -{p{\mu}/2})/(N - ps)$ is the
critical exponent in the sense of Hardy–Littlewood–Sobolev inequality, $f : \RR^N \times \RR  \to \RR$ is a Carath\'{e}odory
function with superlinear growth. Further Liang and R\u{a}dulescu  (\cite{radulescu4}) studied  the existence of infinitely many solutions of the problem \eqref{sck} using  symmetric
mountain pass lemma under some appropriate assumptions on $f$. \\
Motivated by all the above works, 
in problem \eqref{mainprob}, we consider the  
the study of nonlocal Kirchhoff-Choquard type
problems with variable order and variable exponents. Now we fix some notations.
For any domain $\mathcal{D}$ and any function $\Phi:\mathcal{D}\rightarrow\mathbb R$, we set 
{\begin{align*}
	\Phi^{-}:=\inf_{\mathcal{D}} \Phi(x)\text{ ~~~and ~~} \Phi^{+}:=\sup_{ \mathcal{D}}\Phi(x).
	\end{align*}}
\noindent We  define the function space
$$C_+(\mathcal{D}):=\{\Phi\in C(\mathcal{D}, \RR):1 <\Phi^{-}\leq \Phi^{+}<\infty\}.$$
Concerning  the variable order $s$ and the variable exponents $p,\mu$  and the potential function $V$ we assume  the followings:  
\begin{itemize} 
	\item[{\bf(S1)}] $s:\mathbb R^N\times\mathbb R^N\rightarrow\mathbb R$ is a  continuous and symmetric function, i.e., $s(x,y)=s(y,x)$ 
	for all $(x,y)\in \RR^N\times \RR^N$
	with $0<s^-\leq s^+<1$.
	\item[{\bf(P1)}] $p\in C_+(\mathbb R^N\times\mathbb R^N)$ is a continuous and symmetric function, i.e., $p(x,y)=p(y,x)$ for all $(x,y)\in \RR^N\times \RR^N$
	such that $s^+p^+<N.$
	\item[{\bf($\mu$1)}]  $\mu:\R\to\RR$ is a continuous and symmetric function, i.e., $\mu(x,y)=\mu(y,x)$
	for all $(x,y)\in \RR^N\times \RR^N$ with $0<\mu^-\leq \mu^+<N$.
	\item[{\bf(V1)}] $V\in C(\Om,\RR)$ such that $V(x)\geq 0$ for all $x\in\Om.$
\end{itemize}
Through out this article, $p_s^*(x):=\frac{N\p}{N-\ol s(x)\p}$ denotes the Sobolev-type critical exponent, where $\ol s(x):=s(x,x).$
Next, we consider the following assumption on the  Kirchhoff function $m$ in \eqref{mainprob}.
\begin{itemize}\item[{\bf (M1)}]   $m:\RR_0^+\rightarrow \RR_0^+$ is continuous and  defined as $m(t)=a+bt^{\theta-1}, ~a\geq0,b>0$ such that $\theta \in[1,2\frac{p_s^{*-}}{p^+q^+}),$ where $p^+\geq 2$ and $\mu$ satisfies $(\mu1)$ and $q\in C_+(\R)$ verifies 
	{\begin{equation}\label{symetry}
		\frac{2}{q(x,y)} + \frac{\mu(x,y)}{N}=2, \quad \forall x,y \in \RR^N.
		\end{equation}}
\end{itemize}
Here $a=0$ represents the degenerate Kirchhoff equation and $a>0$ represents non-degenerate Kirchhoff equation. 
In case of the degenerate Kirchhoff problems, the transverse
oscillations of a stretched string, with nonlocal flexural rigidity, depends continuously on the Sobolev
deflection norm of $u$ via $m(\cdot)$. The fact $m(0) = 0$ means that the base
tension of the string is zero, a very realistic model from a physical point of view. 
We would like to mention that in this article we study problem \eqref{mainprob} for both the degenerate and non-degenerate cases.
\noi Let $ M(t):=\int_{0}^{t}m(s)ds$ denote the primitive of $m$. Then we have  next two remarks as consequences of $(M1)$.
\begin{remark}\label{main}
	When $a=0,$ we have the following observations:
	\begin{itemize}
		\item For any $\tau>0$, there exists  $m_0:= m_0(\tau)>0$ such that $m(t)\geq m_0$ whenever $t \geq \tau.$
		\item $\theta  M(t)-m(t)t$ is non decreasing for $t>0$ and $\theta  M(t)-m(t)t = 0\;\;\; \text {~for~all~} \; t \geq 0.$
		\item $M(t)=M(1) t^\theta, \text{~~where~} M(1)=\frac{b}{\theta}.$
	\end{itemize}
\end{remark}
\begin{remark}\label{rem-M}When $a>0,$ we have the following observations:
	\begin{itemize}
		\item $m(t)=a+b t^{\theta-1}, a>0$ and $m(t)\geq\inf_{t\geq 0}m(t)=a>0.$
		\item $\theta  M(t)-m(t)t$ is non decreasing for $t>0$ and
		$\theta  M(t)-m(t)t \geq 0\;\;\; \text {~for~all~} \; t \geq 0.$
		\item For each $t \geq 0$
		{ \begin{equation}\label{m4}
			\left\{ \begin{array}{rl}
			M(t)&\geq M(1)t^\theta\; \text{~~~~~~~~~~~for all}\; t \in[ 0,1],\\
			&\leq  M(1)t^\theta\; \text{~~~~~~~~~~~for all}\; t \geq 1,\\
			&\leq M(1)(1+t^\theta)\; \text{~~~~for all}\; t \geq 0,
			\end{array}
			\right.\end{equation}}
		where $M(1)=a+\frac{b}{\theta}.$
	\end{itemize}
\end{remark}
The assumptions that we consider for the nonlinearity $f(x,t)$ in \eqref{mainprob} are as follows:
\begin{itemize}
	\item[{\bf(f1)}] $f\in C(\Om\times\RR,\RR)$ such that  $|f(x,t)|\leq \mathcal{C}\left(1+|t|^{r(x)-1}\right),$ where $\mathcal{C}>0$ is a constant, $r\in C_+(\RR^N)$ satisfies $1< r(x)q^-\leq r(x) q^+<p_s^*(x),$ $r^->\frac{\theta p^+}{2}$ and $q\in C_+(\R)$ verifies \eqref{symetry}.
	\item [{\bf(f2)}]$f(x,t)=o \left({|t|^{\frac{\theta p^+}{2}-2}t}\right)$ as $|t|\to0,$ uniformly in $x\in\Om.$ 
	\item [{\bf(f3)}]$\DD\lim_{|t|\to\infty}\frac{F(x,t)}{|t|^{\gr}}=
	+\infty$ uniformly in $x\in\Om,$  where $F(x,t):=\int_{0}^{t}f(x,s)ds$ is the primitive of $f.$ 
	\item [{\bf(f4)}] There exists $\vartheta>1$ such that $\vartheta\mathcal{F}(x,t)\geq\mathcal{F}(x,st)$ for $(x,t)\in\Om\times\RR$ and $s\in[0,1],$ where $\mathcal{F}(x,t)=2tf(x,t)-{\theta p^+}F(x,t).$
\end{itemize}
The condition $(f4)$ is originally due to Jeanjean \cite{jean} in the case
$p(x)=2$, and then was used in \cite{liu1} for $p-$Laplacian equations in bounded domain. It is worthy to note that the assumptions $(f2)-(f4)$  allow us to consider the nonlinearities which do 
not satisfy following standard Ambrosetti–Rabinowitz {(AR)} type condition:\\
{\bf(AR)} There exists $\omega>\theta p^+$ such that
{$$0< \omega F(x,t)\leq 2tf(x,t), t\not=0, \text{ for all } x\in\Om .$$}
An example of such function is
$f(x,t)=t|t|^{\gr-2}\log(1+|t|).$ 
Ambrosetti–Rabinowitz condition
ensures that an Euler-Lagrangian functional has the mountain pass geometry
structure  and also plays a pivotal role in establishing the boundedness of the Palais-Smale sequence of the functional. Therefore,
relaxing {(AR)} condition not only includes a larger class of nonlinearities but also calls for 
delicate analysis to establish the compactness results 
and hence interests many studies (see for e.g., \cite{alves,radulescu,liu2} and references there in).
We make the following remarks about $f(x,t)$.
\begin{remark}\label{F4}
	Since $f(x,0)=0=F(x,0),$ thanks to $(f2),$  from $(f4)$ we get $\mathcal{F}(x,t)\geq0,$ that is,
	{\begin{equation}\label{f4}
	2tf(x,t)-{\theta p^+}F(x,t)\geq0 \text{~~for all ~~} \rr.
	\end{equation}}
\end{remark}
\noi The following crucial remark  was  studied in \cite{liu2} for local $p-$Laplacian set up.
\begin{remark}\label{f5}
	$F(x,t)\geq 0$        for all $\rr.$
	\begin{proof}
		For $t>0,$ using \eqref{f4}, we have 
	{	$$\frac{d}{dt}\frac{F(x,t)}{t^{\gr}}=\frac{t^{\gr} f(x,t)-\frac{\theta p^+}{2}t^{\gr-1}F(x,t)}{t^{\theta p^+}}=\frac{\mathcal{F}(x,t)}{2\;t^{\gr+1}}\geq0.$$}
		Also from $(f2),$ we can easily deduce that $\DD\lim_{t\to 0^+}\frac{F(x,t)}{t^{\gr}}=0.$
		Using the above two facts, it follows that $F(x,t)\geq0$ for all $(x,t)\in\Om\times\RR,~t\geq0.$\\
		Similarly, for $t<0,$ proceeding  as above, we get 
		$\DD\lim_{t\to 0^-}\frac{F(x,t)}{(-t)^{\gr}}=0,$ and therefore  $F(x,t)\geq0$  for all $(x,t)\in\Om\times\RR,~t\leq0.$
		Thus the result follows.
	\end{proof}
\end{remark}
\noi In view of the above remark, we have the following assertion:
\begin{remark}\label{f6}
	From \eqref{f4} and Remark \ref{f5}, we obtain
	$f(x,t)\geq0 \text{ for all } (x,t)\in\Om\times\RR, ~t\geq0 \text{ and } 
	f(x,t)\leq0 \text{ for all } (x,t)\in\Om\times\RR,~t\leq0.
	$
	Therefore  for all $x \in \Om,$ we have
	$F(x,t)$ { is non decreasing in $t\geq0$}  
	and is non increasing in $t\leq0.$
\end{remark}
\begin{definition}[Weak solution]
	A function $u\in E$ ({defined in section \ref{Pre}}) is said to be weak solution of \eqref{mainprob}, if for all $w\in E$ 
	{	\begin{align*}
	&m(\sigma(u))\Big[\int_{\RR^N}\int_{\RR^N}\frac{|
			u(x)-u(y)|^{p(x,y)-2}(u(x)-u(y))(w(x)-w(y))}{|
			x-y|^{N+s(x,y)p(x,y)}}dxdy\\&~~~~~~~~~~~~~~~~~~~~+\int_{\Om}V(x)|u(x)|^{\p-2}u(x)w(x)dx\Big]=\int_{\Om}\int_{\Om}\frac{F(y,u(y))f(x,u(x))w(x)}{|x-y|^{\mu(x,y)}}dx.\end{align*}}
\end{definition}
\noi The weak solutions are characterized as the critical point of the associated energy functional $J$ (see Definition \ref{energy}). Now
we state the main results to be proved in this article. 
\begin{theorem}\label{mainthm}
	Let $(S1),(P1),(\mu1), (V1)$ and $(M1)$ hold.
	Also let $f$ satisfy $(f1)-(f4)$. Then \ref{mainprob} admits a nontrivial weak solution. 
\end{theorem}
\noi The next theorem  deals with  the  ground state solution  of the problem \eqref{mainprob}, which 
is a solution that
minimizes the  functional $J$ among all nontrivial solutions.
To study ground state solution  for  \eqref{mainprob} without {(AR)} condition, we consider the assumption: 

\begin{itemize}
	\item [{\bf(f4)$'$}]  $\frac{f(x,t)}{|t|^{\gr-2}t}$ is increasing in $t>0$ and decreasing in $t<0$ for all $x\in\Om.$
\end{itemize}
Again we can see that the conditions $(f2),(f3),(f4)'$ are weaker than {(AR)}. 
\begin{theorem}\label{ground-state 2}
	Let  $(S1),(P1),(\mu1), (V1)$ and $(M1)$ hold.
	Also let $f$ satisfy $(f1)-(f3)$ and $(f4)'.$ 
	Then the problem \ref{mainprob} admits a nontrivial ground state solution. 
\end{theorem}
\noi Next, for the odd nonlinearity $f(x,t),$ we state the existence results  of infinitely many solutions using the Fountain theorem and the Dual fountain theorem.
\begin{theorem}\label{infinite-sol}
	Let $(S1), (P1),(\mu1), (V1)$ and $(M1)$ hold. Also let $f$ satisfy $(f1)-(f4)$ with $f(x,-t)=-f(x,t).$ 
	Then the problem \ref{mainprob} has a sequence of nontrivial weak solutions with unbounded energy.
\end{theorem}
\begin{theorem}\label{dual-fount-sol}
	Let $(S1),(P1),(\mu1), (V1)($ and $M1)$ hold. Also let $f$ satisfy $(f1)-(f4)$ with $f(x,-t)=-f(x,t)$.
	Then the problem \ref{mainprob} has a sequence of nontrivial weak solutions with negative critical values converging to zero.
\end{theorem}
\noi In this article we first prove the existence of nontrivial solution using Mountain pass theorem with Cerami condition, then using Nehari manifold we study existence of nontrivial ground state solution and finally existence of infinitely many solutions are achieved by applying Fountain theorem and Dual fountain theorem for the problem \ref{mainprob} which covers both degenerate and non-degenerate cases.
The main feature of the problem \ref{mainprob} is  the presence of both the
nonlocal Kirchhoff and Choquard terms together for which  \eqref{mainprob} remains no longer a point wise
identity and hence it is categorized as a doubly nonlocal problem.  Moreover, due to the involvement of the variable order and variable exponents,
the problem possess  non-homogeneous nature.  These facts induce some further mathematical difficulties in the use of classical
methods of nonlinear analysis. According to best of our knowledge the equation of type \eqref{mainprob}  involving the non-homogeneous operator $(-\Delta)_{p(\cdot)}^{s(\cdot)}$
is studied for the first time in this article. Also we want to mention that the same results will hold in case of local $p(x)-$Lapacian,
which are also new, as per best of our knowledge, in the literature.  The main novelty of this work is that  unlike as in \cite{alves-choquard,rs}(also as in most of the studies regarding Choquard-type problems existed in the literature) we relax the well
known Ambrosetti-Rabinowitz type condition on our nonlinearity $f$ 
and hence we need to carry out some extra careful and delicate analysis to overcome the difficulties and establish the Cerami 
condition so that we can prove the desired results in this article.
\section{Preliminary results and functional settings}\label{Pre}
\noi In this section first we briefly discuss  
some basic properties of
the  variable exponent Lebesgue spaces, which will be used as tools to prove
our main results.\\
For $\Theta\in C_+(\Om)$ define the variable exponent Lebesgue space $L^{\Theta(\cdot)}(\Om)$ as
{$$
L^{\Theta(\cdot)}(\Om) :={ \Big \{ u : \Om\to\mathbb{R}\  \text{is~measurable}: \int_{\Om} |u(x)|^{\Theta(x)} \;dx< \infty \Big \}}
$$}
which is a separable, reflexive, uniformly convex Banach space {(see \cite{diening, fan,radulescu1})} with respect to the Luxemburg norm
{$$
\|{u}\|_{ L^{\Theta(\cdot)}(\Om)}:=\inf\Big\{\eta >0:
\int_{\Om}
\Big|\frac{u(x)}{\eta}\Big|^{\Theta(x)}\;dx\le1\Big\}.
$$ }Define the modular $\rho:\ L^{\Theta(\cdot)}(\Om)\to\mathbb{R}$ as
$
\rho(u):=\int_{\Om }|u|^{\Theta(x)}\; dx, \  { for~all~} u\in
L^{\Theta(\cdot)}(\Om).
$
{\begin{proposition} \label{norm-mod}
	{\rm (\cite{fan})}Let $u_n,u\in L^{\Theta(\cdot)}(\Om)\setminus\{0\},$ then the following properties hold:
	\begin{itemize}
			\item[\rm{(i)}]   $\eta=\|u\|_{ L^{\Theta(\cdot)}(\Om)}$ if and only if  $\rho(\frac{u}{\eta})=1.$
			\item[\rm{(ii)}] $\rho(u)>1$ $(=1;\ <1)$ if and only if  $\|u\|_{ L^{\Theta(\cdot)}(\Om)}>1$ $(=1;\ <1)$,
			respectively.
			\item[\rm{(iii)}] If $\|u\|_{ L^{\Theta(\cdot)}(\Om)}>1$, then $
			\|u\|_{ L^{\Theta(\cdot)}(\Om)}^{\Theta^-}\le \rho(u)\le
			\|u\|_{L^{\Theta(\cdot)}(\Om)}^{\Theta^+}$.
			\item[\rm{(iv)}] If $\|u\|_{ L^{\Theta(\cdot)}(\Om)}<1$, then $
			\|u\|_{L^{\Theta(\cdot)}(\Om)}^{\Theta^+}\le \rho(u)\le
			\|u\|_{L^{\Theta(\cdot)}(\Om)}^{\Theta^-}$.
			\item[\rm(v)] ${\DD\lim_{n\ra \infty} }\| u_{n} - u \|_{ L^{\Theta(\cdot)}(\Om)} =0\iff{\DD\lim_{n\ra \infty}} \rho(u_{n} -u)=0.$
	\end{itemize}
\end{proposition}}
\noi Let $\Theta'$ be conjugate function of $\Theta,  $ that is,  $1/\Theta(x)+1/\Theta'(x)=1$.
\begin{proposition}{\rm (H\"{o}lder inequality)} \label{Holder}{\rm (\cite{fan})}	
	For any $u\in
	L^{\Theta(\cdot)}(\Om)$ and $v\in L^{\Theta'(\cdot)}(\Om)$, we have
	{$$
	\Big|\int_{\Om} uv\,d x\Big|
	\leq
	2\|u\|_{ L^{\Theta(\cdot)}(\Om)}\|{v}\|_{ L^{\Theta'(\cdot)}(\Om)}.
	$$}
\end{proposition}
{\noi The above result is also valid  for  $\Om=\RR^N,$ which is called generalized H{\"o}lder inequality.
{\begin{lemma}{\rm(\cite{sweta})}\label{lemA1}
		Let  $\vartheta_1(x)\in L^\infty(\Om)$ such that $\vartheta_1\geq0,\; \vartheta_1\not\equiv 0.$ Let $\vartheta_2:\Om\ra\RR$ be a measurable function 
		such that $\vartheta_1(x)\vartheta_2(x)\geq 1$ a.e. in $\Om.$ Then for every $u\in L^{\vartheta_1(x)\vartheta_2(x)}(\Om),$ 
		$$\parallel |u|^{\vartheta_1(\cdot)}\parallel_{L^{\vartheta_2(x)}(\Om)}\leq 
		\parallel u\parallel_{L^{\vartheta_1(x)\vartheta_2(x)}(\Om)}^{\vartheta_1^-}+\parallel u\parallel_{L^{\vartheta_1(x)\vartheta_2(x)}(\Om)}^{\vartheta_1^+}.$$
\end{lemma}}
\noi Note that the above lemma also holds if we replace $\Om$ by $\RR^N$.
\par Next we recall the fractional Sobolev spaces with variable order and variable exponents (see {\cite{rs}}). 
First denote $	W^{s(\cdot,\cdot),\ol p(\cdot),p(\cdot,\cdot)}(\Om)=W$ and then define
{\begin{align*}
W:=\Big \{ u\in L^{\overline{p}(\cdot)}(\Om):
	\int_{\Om}\int_{\Om}\frac{| u(x)-u(y)|^{p(x,y)}}{\eta^{p(x,y)}| x-y |^{N+s(x,y)p(x,y)}}dxdy<\infty,
	\text{ for some }\eta>0\Big\}
	\end{align*}}
endowed with the norm
{$$
\|u\|_{W}:=\inf \Big \{\eta>0: \rho_W\left(\frac{u}{\eta}\right) <1 \Big \},
$$}
where $$\rho_{W}(u):=\int_{\Om}\left|u\right|^{\ol{p}(x)}d x+
\int_{\Om}\int_{\Om}
\frac{|u(x)-u(y)|^{p(x,y)}}{|x-y|^{N+s(x,y)p(x,y)}}
\,d xd y$$ is a modular on $W.$
Then, $(W, \|\cdot\|_W)$ is a separable reflexive Banach space  {(see \cite{rs})}. 
On $W$ we also make use of the following norm
$$
|u|_{W}:=\|u\|_{L^{\ol{p}(\cdot)}(\Om)}+[u]_{W},
$$
where the seminorm $[\cdot]_W$ is defined as follows:
{$$[u]_{W}:=\inf \Big\{\eta>0:
\int_{\Om}\int_{\Om}
\frac{|u(x)-u(y)|^{p(x,y)}}{\eta^{p(x,y)}|x-y|^{N+s(x,y)p(x,y)}}
\,d x d y <1 \Big \}.
$$ }Note that $\|\cdot\|_{W}$ and $|\cdot|_{W}$ are equivalent norms on $W$  with the relation
{\begin{equation}\label{equivalency}
\frac{1}{2}\|u\|_{W}\leq |u|_{W}\leq 2\|u\|_{W}, \ \  \text{~for~all~} u\in W.
\end{equation}}
\begin{remark}
	If we substitute $\Om,$ a smooth bounded  domain, by  $\RR^N$, all the above results regarding the   fractional Sobolev spaces with variable order and variable exponents hold.
\end{remark}
\noi We define the subspace ${X}_0$ of $W^{s(\cdot,\cdot),\ol{p}(\cdot),p(\cdot,\cdot)}(\RR^N)$ as 
{$${ X_0}={ X}_0^{s(\cdot,\cdot),\overline{p}(\cdot),p(\cdot,\cdot)}(\Omega):=\{u\in W^{s(\cdot,\cdot),\ol{p}(\cdot),p(\cdot,\cdot)}(\RR^N)\;:\; u=0\;a.e.\; in\;\Omega^c\}$$} which is endowed with the following norm :
{\begin{align*}
\| u\|_{{ X}_0}:=\inf\Big\{\eta>0:\rho_{X_0}\left(\frac{u}{\eta}\right)<1\Big\},
\end{align*}}
where $\rho_{X_0}(u):=
\DD\int_{\RR^N}\int_{\RR^N}
\frac{|u(x)-u(y)|^{p(x,y)}}{|x-y|^{N+s(x,y)p(x,y)}}
\,d xd y$ is a convex modular on $ X_0.$
\noi We have the following embedding result which is studied in \cite{rs}. 
\begin{theorem}\label{prp 3.3}
	Let $\Om$ be a smooth bounded domain in $\RR^N, N\geq2$
	and $s(\cdot,\cdot)$ and $p(\cdot,\cdot)$ satisfy $(S1)$ and $(P1),$ respectively. 
	Then for any  $\gamma\in C_+(\overline{\Om})$  with $1<\gamma(x)< p_s^*(x)$ for all $x\in \overline{\Om}$,
	there exits a constant $C=C(N,s,p,\gamma,\Omega)>0$
	such that for every $u\in{ X}_0$, 
	\begin{equation*}
	\| u \|_{L^{\gamma(x)}(\Omega)}\leq C \| u \|_{{ X}_0}.
	\end{equation*}
	Moreover, this embedding is
	compact.
\end{theorem}

\begin{remark} For $u\in X_0,$ from the proof of the  Theorem \ref{prp 3.3},  we get {$$\|u\|_{X_0}\leq\|u\|_{W^{s(\cdot,\cdot),\ol{p}(\cdot),p(\cdot,\cdot)(\RR^N)}}\leq C\|u\|_{X_0},$$} that is, these two norms are equivalent on $X_0.$ Since  $X_0$ is closed subspace of
	the separable reflexive Banach space $W^{s,\ol{p}(\cdot),p(\cdot,\cdot)}(\RR^N)$ with respect to  $\|\cdot\|_{W^{s,\ol{p}(\cdot),p(\cdot,\cdot)}(\RR^N)},$ we have that $(X_0,\|\cdot\|_{X_0})$ is separable reflexive Banach space.\end{remark}
\noi If $(V1)$ holds true, we define the following space 
{$$E=\Big\{u\in X_0:\int_{\Om}\frac{V(x)|u|^{\p}}{\eta^\p}dx<+\infty, \text {~for some~} \eta>0\Big\}$$  }
equipped with the norm {$$
\|u\|_{E}:=\inf \Big \{\eta>0: \rho_E\left(\frac{u}{\eta}\right) <1 \Big \},
$$}
where {$$\rho_{E}(u):=
\int_{\RR^N}\int_{\RR^N}
\frac{|u(x)-u(y)|^{p(x,y)}}{|x-y|^{N+s(x,y)p(x,y)}}
\,d xd y+\int_{\Om}V(x)\left|u\right|^{\ol{p}(x)}d x$$} defines a convex modular in $E.$
{On $E$ we also can make use of the  the following norm
	{$$
	|u|_{E}:=[u]_V+\|u\|_{X_0},
	$$} where {$$[u]_V=\DD\inf\Big\{\eta>0:\int_\Om V(x)\frac{|u(x)|^{\p}}{\eta^{\p}}dx<1\Big\}.$$ }One can easily verify that $\|\cdot\|_{E}$ and $|\cdot|_{E}$ are equivalent norms on $E$  with the relation
	\begin{equation}\label{norm-equiv}
	\frac{1}{2}\|u\|_{E}\leq |u|_{E}\leq 2\|u\|_{E}, \ \  \text{~for~all~} u\in E.
	\end{equation}
\noi	Next we can prove the following results similarly as in {\cite{fan}.}
{\begin{proposition} \label{norm-modular} For $u\in E\setminus\{0\},$ we have 
		\begin{enumerate}
			\item[\rm{(i)}]   $\eta=\|u\|_{E}$ if and only if \ $\rho_E(\frac{u}{\eta})=1;$
			\item[\rm{(ii)}] $\rho_E(u)>1$ $(=1;\ <1)$ if and only if \ $\|u\|_{E}>1$ $(=1;\ <1)$,
			respectively;	
			\item[\rm{(iii)}] if $\|u\|_{E}\geq 1$, then $
			\|u\|_{E}^{p^-}\le\rho_E(u)\le
			\|u\|_{E}^{p^+}$;
			\item[\rm{(iv)}] if $\|u\|_{E}<1$, then $
			\|u\|_{E}^{p^+}\le \rho_E(u)\le
			\|u\|_{E}^{p^-}$.
		\end{enumerate}
\end{proposition}}

\noi As a consequence of the above proposition, we can derive the following result.
\begin{proposition}\label{norm-convergence}
Let $u,u_{n}  \in E,n\in\mathbb N.$
Then	the following statements are equivalent:
	{	\begin{enumerate}
						\item[\rm(i)] 
			$\DD{\lim_{n\ra \infty} }\| u_{n} - u \|_{ E} =0.$ 
			\item[\rm(ii)] 
			$\DD{\lim_{n\ra \infty}} \rho_E(u_{n} -u)=0.$
	\end{enumerate}}
\end{proposition}	
	\begin{lemma}$(E,\|\;\cdot\;\|_E)$ is a separable reflexive Banach space.
	\end{lemma}
	
	\begin{proof}
		First we show that $(E,\|\;\cdot\;\|_E)$
		is a Banach space. For that,	let $\{u_n\}$ be a Cauchy sequence in $E.$ Therefore for any $\e>0$ there exists $N_\e\in\mathbb N$ such that if $n,k\geq N_\e$ \begin{align}\label{0.1}
		\|u_n-u_k\|_E\leq \e.
		\end{align}
		Since $\|u\|_E\geq\|u\|_{X_0}$ and $(X_0,\|\;\cdot\;\|_{X_0})$ is a Banach space, there exists $u\in X_0$ such that $u_n\to u$ in $X_0$ strongly as $\gt.$ So, there exists a subsequence $\{u_{n_j}\}$ such that $u_{n_j}(x)\to u(x)$ a.e. $x\in\RR^N.$ Now using Fatou's lemma and \eqref{0.1} with $\e=1,$ we have
		{\begin{align}
		\int_{\Om}V(x)|u(x)|^\p dx&\leq \DD\liminf_{n\to\infty}\int_{\Om}V(x)|u_n(x)|^\p dx\n
		&\leq \DD\liminf_{n\to\infty}\int_{\Om}V(x)|u_n(x)-u_{N_1}(x)-u_{N_1}(x)|^\p dx\n
		&\leq 2^{p^+}\DD\liminf_{n\to\infty}\Big[\int_{\Om}V(x)|u_n(x)-u_{N_1}(x)|^\p+\int_{\Om}V(x)|u_{N_1}(x)|^\p dx\Big]\n
		&\leq 2^{p^+}\Big[1+\int_{\Om}V(x)|u_{N_1}(x)|^\p dx\Big]<\infty.
		\end{align}}
		Therefore $u\in E.$ 
		Now again by Fatou's lemma and \eqref{0.1}, we get for all $n,n_j\geq N_\e$ \begin{align}
		\rho_E (u_n-u)\leq \DD\liminf_{j\to\infty} \rho_E (u_n-u_{n_j})\leq\e
		\end{align} and then Proposition \ref{norm-convergence} infers $u_n\to u.$ Hence  $(E,\|\;\cdot\;\|_E)$ is a Banach space. For proving reflexivity of $E$ we define the  map  
$T:E\to L^{\ol p(\cdot)}(\Om)\times L^{ p(\cdot,\cdot)}(\R)$ as {$${ T(u)=\bigg(V^{1/\p}u,~ \frac{|u(x-u(y)|}{|x-y|^{s(x,y)+\frac{N}{p(x,y)}}}\bigg).}$$}
	The norm on $L^{\ol p(\cdot)}(\Om)\times L^{ p(\cdot,\cdot)}(\R)$ is given as $$\|u\|=\|u\|_{L^{\ol p(\cdot)}(\Om)}+\|u\|_{L^{ p(\cdot,\cdot)}(\R)}.$$ Clearly $T$ is an isometry. Hence $T(E)$ is reflexive being a closed subspace of the reflexive Banach space $L^{\ol p(\cdot)}(\Om)\times L^{ p(\cdot,\cdot)}(\R)$ 
	 (see { \cite[Proposition 3.20]{brezis}}) and consequently $E$ is reflexive. 
	Arguing similarly, we get $E$ is separable (see \cite[Proposition 3.25]{brezis}). 
	\end{proof}

\noi Using Theorem \ref{prp 3.3} and the fact $\|u\|_E\geq\|u\|_{X_0},$ we  have the following embedding theorem.
\begin{theorem}\label{cpt}Let $\Om\subset\RR^N, N\geq2$ be a smooth bounded domain and let
	$(S1),(P1)$ and $(V1)$ hold. 	Then for any  $\gamma\in C_+(\overline{\Om})$  with $1<\gamma(x)< p_s^*(x)$ for all $x\in \overline{\Om}$,
	there exits a constant $C=C(N,s,p,\gamma,\Omega)>0$
	such that for every $u\in E$, 
	\begin{equation*}
	\| u \|_{L^{\gamma(x)}(\Omega)}\leq C \| u \|_{E}.
	\end{equation*}
	Moreover, this embedding is
	compact.
\end{theorem}
\noi The function space $E$ is the solution space for the problem \eqref{mainprob}.
 The energy
functional $J$ associated to \eqref{mainprob} is defined as follows. For $u\in E$, we set
{\begin{itemize}
		\item $\DD\sigma(u):={ \A+\B.}$
		\item $\DD\Psi(u):={ \frac{1}{2}\F.}$
		\item $\DD I(u):={ \f.}$
\end{itemize}}
\begin{definition}\label{energy}
	The energy functional  $J:E\ra\RR$ associated with problem \ref{mainprob}
	is defined as 
	$$J(u)=M(\sigma(u))-\Psi(u).$$
\end{definition}
\noi	To establish the smoothness of the energy functional $J$  we first recall the following  Hardy-Littlewood-Sobolev type result 
( \cite[{Proposition 4.1}]{rs}).
\begin{proposition}\label{HLS} Let $(\mu1)$ hold and let $q \in C_{+}(\R)$ be a 
	continuous function satisfying 
	\begin{equation*}
	\frac{2}{q(x,y)} + \frac{\mu(x,y)}{N}=2, \quad \text { for all } x,y \in \mathbb{R}^N.
	\end{equation*}
	If $h,g\in L^{q^-}(\mathbb{R}^N) \cap L^{q^+}(\mathbb{R}^N)$
	then
{	\begin{align*}
	\displaystyle \bigg| \int_{\mathbb{R}^N}\int_{\mathbb{R}^N} \frac{h(x) g(y)}{|x-y|^{\mu(x,y)}} dx dy \bigg| \leq  C(N,q,\mu)\bigg(\| h\|_{L^{q^{+}}(\mathbb{R}^N)}\| g\|_{L^{q^{-}}(\mathbb{R}^N)} + \| h\|_{L^{q^{-}}(\mathbb{R}^N)})\| g\|_{L^{q^{+}}(\mathbb{R}^N)}\bigg).
	\end{align*}}
\end{proposition}
\begin{corollary}\label{cor1}
	In particular for 
	$h(x)=g(x)=|u(x)|^{\ba(x)} \in L^{q^-}(\Om) \cap L^{q^+}(\Om),$  {$u\in E,$} we have 
{	$$
	\begin{array}{l}
	\displaystyle \left| \int_{\Om}\int_{\Om} \frac{|u(x)|^{\ba(x)}|u(y)|^{\ba(y)}}{|x-y|^{\mu(x,y)}} dx dy \right| \leq  C(N,q,\mu,\beta)\left(\| |u|^{\ba(\cdot)}\|^{2}_{L^{q^{+}}(\Om)} + \| |u|^{\ba(\cdot)}\|^{2}_{L^{q^{-}}(\Om)}\right),
	\end{array}
	$$}
	where $\beta\in   C_+(\ol\Om)$ such that $1< \ba^- q^-\leq\ba(x) q^{-} \leq \ba(x)q^{+} < p_s^{*}(x), $ for all $x \in \ol \Om.$
\end{corollary}
{\begin{remark}From Proposition \ref{HLS} we can define the variable order and variable exponent Hardy-Littlewood-Sobolev critical exponent as
		{$$p_{s,\mu}^*(x):=\frac{p_s^*(x)}{\ol q(x)}=\frac{\p}{2}\left(\frac{2N-\ol\mu(x)}{N-\ol s(x)\p}\right),$$} where $\ol q(x)=q(x,x)$ and $\ol\mu(x)=\mu(x,x).$\end{remark}}

\begin{lemma} 
	The functional $J$ as defined in the Definition \ref{energy} is of class $C^1$ and for all $u,w\in E.$
	{\begin{align*}
	\langle J'(u),w\rangle&=m(\sigma(u))\bigg[\int_{\RR^N}\int_{\RR^N}\frac{|
		u(x)-u(y)|^{p(x,y)-2}(u(x)-u(y))(w(x)-w(y))}{|
		x-y|^{N+s(x,y)p(x,y)}}dxdy\n&~~~~+\int_{\Om}V(x)|u(x)|^{\p-2}u(x)w(x)dx\bigg]-\int_{\Om}\int_{\Om}\frac{F(y,u(y))f(x,u(x))w(x)}{|x-y|^{\mu(x,y)}}dxdy.
	\end{align*}}
\end{lemma}
\begin{proof}
	Clearly, $J$ is well defined. Also it is easy to see that $M(\sigma(\cdot))$ is Gateaux-differentiable in $E$ and the derivative function at $u\in E$ is given as {\begin{align*}
	&m(\sigma(u))\Big[\int_{\RR^N}\int_{\RR^N}\frac{|
			u(x)-u(y)|^{p(x,y)-2}(u(x)-u(y))(w(x)-w(y))}{|
			x-y|^{N+s(x,y)p(x,y)}}dxdy+\int_{\Om}V|u|^{\p-2}uwdx\Big]
		\end{align*}}
	for all $w\in E.$ Let $\{u_n\}$ be a sequence in $E$ such that $u_n\to u$ strongly in $E$ as $\gt.$ Thus  $u_n(x)\to u(x) $ a.e. in $\RR^N.$ Let $p'$ and $\ol p'$ denote the conjugate of $p$ and $\ol p,$ respectively. Then the sequences {$\Big\{\frac{|
		u_n(x)-u_n(y)|^{p(x,y)-2}(u_n(x)-u_n(y))}{|
		x-y|^{(N+s(x,y)p(x,y))/{p'(x,y)}}} \Big\}$ and $\left\{[V(x)]^{1/{\ol{p}'(x)}}|u_n(x)|^{\p-2}u_n(x)\right\}$} are bounded in $L^{p'(\cdot,\cdot)}(\R)$ and in $L^{\ol{p}'(\cdot)}(\Om)$, respectively and as $\gt$ 
	{\begin{align*}\mathcal{U}_n(x,y)&:=\frac{|
			u_n(x)-u_n(y)|^{p(x,y)-2}(u_n(x)-u_n(y))}{|
			x-y|^{(N+s(x,y)p(x,y))/{p'(x,y)}}}\n&~~~~\to\mathcal{U}(x,y):=\frac{|
			u(x)-u(y)|^{p(x,y)-2}(u(x)-u(y))}{|
		x-y|^{(N+s(x,y)p(x,y))/{p'(x,y)}}} \text{~ for a.e. } x,y\in\RR^N,\end{align*} }
and	{$$V(x)^{1/{\ol{p}'(x)}}|u_n(x)|^{\p-2}u_n(x) \to V(x)^{1/{\ol{p}'(x)}}|u(x)|^{\p-2}u(x) \text{ for a.e. }  x\in\Om.$$}
	Thus by { \cite[Proposition 5.4.7]{willem1}},
	we get as $\gt$ {$$\mathcal{U}_n\rightharpoonup\mathcal{U}\text{~~weakly in $L^{p'(\cdot,\cdot)}(\R)$}$$ } and
	{$$V(\cdot)^{1/{\ol{p}'(\cdot)}}|u_n(\cdot)|^{\ol p(\cdot)-2}u_n(\cdot) \rightharpoonup V(\cdot)^{1/{\ol{p}'(\cdot)}}|u(\cdot)|^{\ol p(\cdot)-2}u(\cdot)\text{ ~~weakly in $L^{\ol{p}'(\cdot)}(\Om)$}.$$}  Hence for any $w\in E,$ by Theorem \ref{cpt} and  definition of weak convergence, 
	{\begin{align}\label{01}
		&\DD\lim_{n\ra \infty}\int_{\RR^N}\int_{\RR^N}\frac{|
			u_n(x)-u_n(y)|^{p(x,y)-2}(u_n(x)-u_n(y))(w(x)-w(y))}{|
			x-y|^{N+s(x,y)p(x,y)}}dxdy\n&~~~~~=\int_{\RR^N}\int_{\RR^N}\frac{|
			u(x)-u(y)|^{p(x,y)-2}(u(x)-u(y))(w(x)-w(y))}{|
			x-y|^{N+s(x,y)p(x,y)}}dxdy,\n
&	\DD\lim_{n\ra \infty}\int_{\Om}V(x)|u_n(x)|^{\p-2}u_n(x)w(x)dx=\int_{\Om}V(x)|u(x)|^{\p-2}u(x)w(x)dx.
	\end{align}}
	Next, using Proposition \ref{HLS} and arguing similarly as in \cite[Section 3]{alves-choquard}, one can see  that $\Psi$ is of class $C^1$ such that Gateaux-derivative of $\Psi$ is given
	as {$$\Psi'(u)w=\int_{\Om}\int_{\Om}\frac{F(y,u)}{|x-y|^{\mu(x,y)}}f(x,u)w(x)dxdy$$} for all $w\in E.$  Moreover
	{\begin{align}\label{03}
		\int_{\Om}\int_{\Om}\frac{F(y,u_n)}{|x-y|^{\mu(x,y)}}f(x,u_n)w(x)dxdy\to \int_{\Om}\int_{\Om}\frac{F(y,u)}{|x-y|^{\mu(x,y)}}f(x,u)w(x)dxdy \text{~~as~}\gt.
		\end{align}}
	Finally combining \eqref{01}-\eqref{03}, we obtain as $\gt$ 
	{$$\|J'(u_n)-J'(u)\|_{E^*}=\DD\sup_{w\in E,\|w\|_E=1} |\langle J'(u_n)-J'(u). w\rangle|\to 0.$$} This completes the proof. 
\end{proof}
\section{Proof of the main results}
\noi Here we give the proofs of  the main theorems in this article.
{From now on wards  $C$ is treated as a generic positive constant which may vary from line to line.} The notation $o_n(1)$ implies $\lim_{n\to+\infty} o_n(1)=0.$
\subsection{Proof of Theorem \ref{mainthm}} 

\noi To prove the Theorem \ref{mainthm} we need some Lemmas and results presented below.
\begin{lemma}[Mountain Pass Geometry 1]\label{mp1}
	Assume that $(S1), (P1),(\mu1),$ $(V1),$  $(M1)$ and $(f1)-(f4)$ hold. Then there exist $R>0$ and $\delta>0$ such that $J(u)>R$ for all $u\in E$ with $\|u\|_E=\delta.$
\end{lemma}
\begin{proof}
	First  we estimate $\Psi(u).$ For that, using $(f1)$ and Theorem \ref{cpt}, one can easily check that $F(\cdot,u)\in L^{q^-}(\Om)\cap L^{q^+}(\Om).$ 
	Hence by {Proposition \ref{HLS}}, we get
	{\begin{align}\label{2}
	\Psi(u)=\frac{1}{2}\F\leq C(N,q,\mu)\Big[\|F(\cdot,u)\|_{L^{q^+}(\Om)}^2+\|F(\cdot,u)\|_{L^{q^-}(\Om)}^2\Big].
	\end{align}}
	From $(f1)$ and $(f2)$ we deduce that for any $\epsilon>0,$ there exist some constants $\ol{C}{(\epsilon)},~C(\epsilon)$ $>0$ such that 
	$|f(x,t)t|\leq \epsilon\gr |t|^{\frac{\theta p^+}{2}-1}+\ol{C}(\epsilon) |t|^{r(x)-1}$ for a.e. $x \in\Om$ and for all $t\br$ 
and
	{\begin{align}\label{3}
	|F(x,t)|\leq \epsilon |t|^{\frac{\theta p^+}{2}}+C(\epsilon) |t|^{r(x)} \text{~~ for a.e. $x \in\Om$ and for all $t\br$}.
	\end{align} }
	Using \eqref{3} and { Lemma \ref{lemA1} }, we have
	{\begin{align}\label{4}
		\|F(\cdot,u)\|_{L^{q^+}(\Om)}&\leq\Big[\DD\int_{\Om}\Big(\epsilon |u(x)|^{\frac{\theta p^+}{2}}+C(\epsilon) |u(x)|^{r(x)}\Big)^{q^+}dx\Big]^{1/{q^+}}\n
		&\leq 2\Big[\epsilon\Big(\int_{\Om} |u(x)|^{\frac{\theta p^+}{2}q^+}dx\Big)^{1/{q^+}}+C(\epsilon) \|~|u|^{r(\cdot)}~\|_{L^{q^+}(\Om)}\Big]\n
		&\leq 2\Big[\epsilon\|u\|_{L^{\gr q^+}(\Om)}^{\frac{\theta p^+}{2}}
		+C(\epsilon) \Big\{\|u\|_{L^{r(\cdot)q^+}(\Om)}^{r^+}+\|u\|_{L^{r(\cdot)q^+}(\Om)}^{r^-}\Big\}\Big].
		\end{align}}
	Similarly, we deduce
	{\begin{align}\label{5}
	\|F(\cdot,u)\|_{L^{q^-}(\Om)}\leq 2\Big[\epsilon\|u\|_{L^{\gr q^-}(\Om)}^{\gr}
	+C(\epsilon) \Big\{\|u\|_{L^{r(\cdot)q^-}(\Om)}^{r^+}+\|u\|_{L^{r(\cdot)q^-}(\Om)}^{r^-}\Big\}\Big]
	\end{align}}
	Plugging \eqref{4} and \eqref{5} into \eqref{2}, we derive
{	\begin{align}\label{5.1.1}
	\Psi(u)&\leq C_h\Big[\Big\{\epsilon^2\|u\|_{L^{\gr q^+}(\Om)}^{\theta p^+}
	+C(\epsilon)^2 \Big(\|u\|_{L^{r(\cdot)q^+}(\Om)}^{2r^+}+~\|u\|_{L^{r(\cdot)q^+}(\Om)}^{2r^-}\Big)\Big\}\n
	&~~~~~~~~~~~~~+\Big\{\epsilon^2\|u\|_{L^{\gr q^-}(\Om)}^{\theta p^+}
	+C(\epsilon)^2 \Big(\|u\|_{L^{r(\cdot)q^-}(\Om)}^{2r^+}+\|u\|_{L^{r(\cdot)q^-}(\Om)}^{2r^-}\Big)\Big\}\Big],
	\end{align}}
where $C_h>0$ is a constant. Now  by applying Theorem \ref{cpt} in \eqref{5.1.1}, we have
{	\begin{align}\label{5.1}
	\Psi(u)\leq C\left[\epsilon^2\|u\|_{E}^{\theta p^+}
	+C(\epsilon)^2 \left\{\|u\|_{E}^{2r^+}+~\|u\|_{E}^{2r^-}\right\}\right].
	\end{align}}
Let $u\in E,~\|u\|_E<1.$	Therefore using {\eqref{5.1}, Remark \ref{main}{(or Remark \ref{rem-M} ) } and Proposition \ref{norm-modular}}, we get
	{\begin{align}\label{1}
	J(u)
	&\geq M(1)\{\sigma(u)\}^\theta-C\left[\epsilon^2\|u\|_{E}^{\theta p^+}
	+C(\epsilon)^2 \left\{\|u\|_{E}^{2r^+}+~\|u\|_{E}^{2r^-}\right\}\right]\n
	&\geq \frac{M(1)}{(p^+)^{\theta}}\{\rho_E(u)\}^\theta-C\epsilon^2\|u\|_{E}^{\theta p^+}
	-2CC(\epsilon)^2 \|u\|_{E}^{2r^-}\n
	&\geq\frac {M(1)}{(p^{+})^{\theta}}\|u\|_E^{\theta p^+}-C\epsilon^2\|u\|_{E}^{\theta p^+}
	-2CC(\epsilon)^2 \|u\|_{E}^{2r^-}
	\end{align}}
	By taking $0<\DD\epsilon<\left[ {M(1)}/{(2C(p^{+}){\theta})}\right]^{1/2}$ in \eqref{1},
	since $\|u\|_E<1$ and $\gr<r^-$ we can choose $0<\delta<1$ sufficiently small such that   \eqref{1} infers that there exists some $R>0$ such that $J(u)>R>0$ for $\|u\|_E=\delta.$ 
\end{proof} 

\begin{lemma}[Mountain Pass Geometry 2]\label{mp2}
	Assume that $(S1),(P1),(\mu1),$ $(V1),$  $(M1)$ and $(f1)-(f4)$ hold. Then there exists $e\in X_{0}$ with $\| e \|_{X_0}>\delta$ such that $J(e)<0,$ where $\delta$ is given by Lemma \ref{mp1}.
\end{lemma} 
\begin{proof}
	Choose $u\in E, u>0$ such that $\|u\|_E=1$ and $\int_{\Om}\int_{\Om}\frac {|u(x)|^{\gr}|u(y)|^{\gr}}{|x-y|^{\mu(x,y)}}dxdy>0.$ Now for $t>1$ large, using  Remark \ref{main} {(or Remark \ref{rem-M} ) } and {Proposition \ref{norm-modular}}, we get
	{\begin{align}\label{6}
	J(tu)
	&\leq \frac{M(1)}{(p^-)^\theta}\left(\rho_E(tu)\right)^\theta-\Psi(tu)\n
	&\leq \frac{M(1)}{(p^-)^\theta}t^{\theta p^+}\left(\rho_E(u)\right)^\theta-\Psi(tu)=\frac{M(1)}{(p^-)^\theta}t^{\theta p^+}-\Psi(tu).
	\end{align}}
	It follows from $(f3)$ that for any $l>0$ there exists $C_l>0$ such that$$F(x,tu(x))>l |tu(x)|^{\gr},$$ whenever $|tu(x)|>C_l$ for a.e. $x\in \Om.$	
	Therefore, using the above inequality in \eqref{6}, we  deduce that
{	\begin{align}\label{7}
	J(tu)\leq \frac{M(1)}{(p^-)^\theta}t^{\theta p^+}-\frac{l^2}{2}t^{\theta p^+}\Big(\int_{\Om}\int_{\Om}\frac {|u(x)|^{\gr}|u(y)|^{\gr}}{|x-y|^{\mu(x,y)}}dxdy\Big).
	\end{align}} After taking $ 0<l<\big[4~\frac{M(1)}{(p^-)^\theta}\big(\int_{\Om}\int_{\Om}\frac {|u(x)|^{\gr}|u(y)|^{\gr}}{|x-y|^{\mu(x,y)}}dxdy\big)^{-1}\big]^{1/2}$ in \eqref{7} we can choose $t_*>0$ large enough so that $|t_*u(x)|>C_l$ for a.e. $x\in \Om$ with $\|t_*u\|_E>\delta $ such that $J(t_*u)<0.$  Thus  by fixing $e=t_*u,$ the result follows.
\end{proof}
\begin{definition}[Cerami condition]\label{cerami}
	$J$ is said to be satisfying Cerami condition $(C)_c$ for any $c\in\RR,$ if for any sequence $\{u_n\}$ in $E$ \begin{align*}
	(C)_c~~~~~~~~~~~~~~~~~~~~~~~~J(u_n)\to c \text{~~~~~ and ~~~~~} (1+\|u_n\|_E)\|J'(u_n)\|_{E^*}\to 0, \text { as } n\to\infty
	\end{align*}
	then $u_n\to u$ strongly in $E$ as $\gt.$
\end{definition}
\begin{lemma}\label{bounded}
	Let $(S1),(P1),(\mu1),$ $(V1),$  $(M1)$ and $(f1)-(f4)$ hold.  Then  the functional $J$ satisfies the Cerami condition $(C)_c$ for any $c\in\RR.$
\end{lemma}
\begin{proof}
	Let $\{u_n\}\subset E$ be a Cerami sequence for $J$ at level $c\in\RR.$ Then by Definition \ref{cerami}, 
	\begin{align}\label{c1}
	J(u_n)\to c \text{~~~~~ and ~~~~~} (1+\|u_n\|_E)\|J'(u_n)\|_{E^*}\to 0 \text { as } n\to+\infty,
	\end{align}
which implies that
	\begin{align}\label{c2}
	\langle J'(u_n),u_n\rangle\to 0 \text{~~ as } n\to +\infty,
	\end{align}
	where $\langle\cdot,\cdot\rangle$ denotes the duality pairing between $E$ and its dual $E^*.$\\
	{First we discuss  the degenerate case, i.e., $a=0.$} Hence we divide the proof into two parts.\\
	{{\bf{Case}:$\DD\inf_{n\in\mathbb N} \|u_n\|_E=d_*>0.$}} First we prove that the sequence $\{u_n\}$ is bounded in $E.$ Indeed, arguing
	by contradiction, we assume that $\{u_n\}$ is unbounded in $E,$ that is, 
	\begin{align}\label{c3}
	\|u_n\|_E\to+\infty \text{~~ as ~} n\to+\infty.
	\end{align} 
	Without loss of generality, we assume $\|u_n\|_E>1$ and  
	$w_n\rightharpoonup w$ weakly in   $E$ and $w_n(x)\to w(x)$ a.e. $x \in \RR^N$ 
	and thus by applying { Theorem \ref {cpt},} it follows that
	\begin{align}\label{c5}
	w_n\to  w \text{  strongly in  } L^{\gamma(\cdot)}(\Om) \text{~ for any~ } 1< \gamma(x)<p_s^*(x) \text{ as } n\to +\infty.
	\end{align}
	Let $\Om_0:=\{x\in\Om: w(x)\not=0\}.$ Thus we have 
	\begin{align}\label{c6}
	|u_n(x)|\to+\infty \text{ ~a.e. }  x \in\Om_0 \text{~~as } n\to+\infty.
	\end{align}
	When $x\in\Om_0,$  we have $|w_n(x)|>0$ for large $n.$ Therefore using this fact together with $(f3)$ and \eqref{c6},  for each $x\in\Om_0$ and sufficiently large $n,$ we get 
	\begin{align}\label{c7}
	\DD\lim_{|u_n(x)|\to\infty}\frac{F(x,u_n(x))}{|u_n(x)|^{\gr}}|w_n(x)|^{\gr}=+\infty.
	\end{align}
	{Now using Remark \ref{f5}, \eqref{c7} and Fatou's lemma, we derive
		{\begin{align}\label{c8}
		\DD\liminf_{n\to\infty}\int_{\Om}\frac{F(y,u_n(y))|w_n(y)|^\gr}{|x-y|^{\mu(x,y)}|u_n(y)|^\gr}dy
		&\geq\DD\liminf_{n\to\infty}\int_{\Om_0}
		\frac{F(y,u_n(y))|w_n(y)|^\gr}{|x-y|^{\mu(x,y)}|u_n(y)|^\gr}dy\n
		&\geq\int_{\Om_0}\DD\liminf_{n\to\infty}
		\frac{F(y,u_n(y))|w_n(y)|^\gr}{|x-y|^{\mu(x,y)}|u_n(y)|^\gr}dy \n&=+\infty.
		\end{align}}
		Combining together \eqref{c7} and \eqref{c8} for each $x\in \Om_0,$ we obtain 
	{\begin{align*}
		\Big(\int_{\Om}\frac{F(y,u_n(y))|w_n(y)|^\gr}{|x-y|^{\mu(x,y)}|u_n(y)|^\gr}dy\Big
		)\frac{F(x,u_n(x))}{|u_n(x)|^{\gr}}|w_n(x)|^{\gr}\to+\infty 
		\text{ as } \gt,
		\end{align*} }
		that is, 
		{ \begin{align}\label{c9}
		\frac{1}{\|u_n\|_E^{\theta p^+}}\Big(\int_{\Om}\frac{F(y,u_n(y))}{|x-y|^{\mu(x,y)}}dy\Big
		){F(x,u_n(x))}\to+\infty
		\text{ as } \gt.
		\end{align} }
		
		\noi We claim that $meas(\Om_0)=0.$ Indeed, if not, then for large $n$ taking $\|u_n\|_E>1$ and using\eqref{c1}, \eqref{c3}, \eqref{c9} and  Remark \ref{main}, Remark\ref{f5} with Fatou's lemma, we get
		{\begin{align*}
		\frac{1}{(p^-)^\theta}&\geq\DD\liminf_{n\to\infty}\frac{1}{(p^-)^\theta} \frac{[\rho_E(u_n)]^\theta}{\|u_n\|_E^{\theta p^+}}\n
		&\geq\DD\liminf_{n\to\infty} \frac{[\sigma(u_n)]^\theta}{\|u_n\|_E^{\theta p^+}}\n
		&=\DD\liminf_{n\to\infty}\frac{1}{\|u_n\|_E^{\theta p^+}} \frac{\left[J(u_n)+\Psi(u_n)\right]}{M(1)}\n
		&\geq \DD\liminf_{n\to\infty}\frac{\Psi(u_n)}{M(1)\|u_n\|_E^{\theta p^+}}-1\n
		&\geq\frac{1}{2M(1)}\DD\liminf_{n\to\infty}\int_{\Om_0}\frac{1}{\|u_n\|_E^{\theta p^+}}{\left(\DD\int_{\Om}\frac{F(y,u_n(y))}{|x-y|^{\mu(x,y)}}dy\right
			){F(x,u_n(x))}}dx -1\n
		&\geq\frac{1}{2M(1)}\int_{\Om_0}\DD\liminf_{n\to\infty}\frac{1}{\|u_n\|_E^{\theta p^+}}{\left(\DD\int_{\Om}\frac{F(y,u_n(y))}{|x-y|^{\mu(x,y)}}dy\right
			){F(x,u_n(x))}}dx -1=+\infty,
		\end{align*}}}
	which is a contradiction and hence $meas (\Om_0)=0.$ Therefore
	\begin{align}\label{c13}
	w(x)=0
	\text{~ a.e. } x\in\Om.\end{align}
	Given any real number $\kappa>1,$ by $(f1),$ it follows that $$F(x,\kappa t)\leq C\left(|\kappa t|+|\kappa t|^{r(x)}\right) \text{ for any } x\in\Om\text{ and for all } t\in\RR,$$
	which together with Theorem \ref{cpt} yields that $$|F(x,\kappa w_n(x)|^{q^+}\leq\overline{h}(x) ,~~~ |F(x,\kappa w_n(x)|^{q^-}\leq\underline{h}(x) \text {~~a.e.~}   x\in\Om\text {  for some ~} \overline{h},~\underline{h}\in L^1(\Om).$$
	Note that from \eqref{c13} we have $w_n\to 0$ strongly in $L^{\gamma(\cdot)}(\Om),$ for all $1<\gamma(x)<p_s^*(x)$ and $w_n(x)\to 0$ a.e. in $\Om,$ hence using the continuity of $F,$ we deduce
	{$$\DD\lim_{n\to+\infty}F(x,\kappa w_n(x))=F(x,0)=0 \text{~ a.e.~} x\in\Om .$$}
	Therefore by Lebesgue dominated convergence theorem, we have 
	{\begin{align}\label{c14}
	\DD\lim_{n\to+\infty}\|F(\cdot,\kappa w_n(\cdot))\|_{L^{q^+}(\Om)}=0 \text{ and }\DD\lim_{n\to+\infty}\|F(\cdot,\kappa w_n(\cdot))\|_{L^{q^-}(\Om)}=0.
	\end{align}}
	Using {{Proposition \ref{HLS}}} and \eqref{c14},  we get as $\gt$
{	\begin{align}\label{c15}
	\Psi(\kappa w_n)\leq C(N,q,\mu)	\Big[\|F(\cdot,\kappa w_n(\cdot))\|_{L^{q^+}(\Om)}^2+\|F(\cdot,\kappa w_n(\cdot))\|_{L^{q^-}(\Om)}^2\Big]\to 0 .
	\end{align}}
	{As $J(tu_n)$ is continuous in $t\in[0,1],$ for each $n\in\mathbb N,$ there exists $t_n\in [0,1]$ such that }
	{\begin{align}\label{c16}
	J(t_nu_n)=\DD\max_{t\in[0,1]}J(tu_n).
	\end{align}}
	We claim that 
	{\begin{align}\label{c17}
	J(t_nu_n)\to+\infty \text{ as } n\to+\infty.
	\end{align} }
	For any real number $C>1,$ choose $\kappa =[{C}/{(\min\big\{1, \frac{m_0}{\theta p^+} \big\})}]^{1/{p^-}}.$ Using \eqref{c3},  we have $\frac{\kappa }{\|u_n\|_E}\in(0,1)$ for $n$  sufficiently large.
	Thus by $(M1),$ Remark \ref{main},\eqref{c15}, \eqref{c16} and Proposition \ref{norm-modular}, we get
{	\begin{align*}
	J(t_nu_n)\geq J\left(\frac{\kappa }{\|u_n\|_E}u_n\right)=J(\kappa w_n)
	&=M(\sigma(\kappa  w_n))-\Psi(\kappa  w_n)\n
	&= \frac{1}{\theta}m(\sigma(\kappa  w_n))\sigma(\kappa  w_n)+o_n(1)\n
	&\geq \frac{m_0}{\theta p^+}  \rho_E(\kappa  w_n)+o_n(1)\n
	&\geq \frac{m_0}{\theta p^+}  (\kappa ) ^{p^-}+o_n(1)\geq C +o_n(1).
	\end{align*}}
\noi Hence \eqref{c17} is proved.
	Because $J(0)=0$ and $J(u_n)\to c$ as $\gt,$ we can see that 
	{\begin{align}\label{c18}
	t_n\in (0,1) \text{ ~and~}\langle J'(t_nu_n), t_nu_n\rangle=t_n ~\frac{d}{dt}\Big|_{t=t_n}J(tu_n)=0.
	\end{align}}
	Combining  Remark \ref{main}, \ref{f5}, \ref{f6}  with  $(f4)$, \eqref{f6}, \eqref{c1}, \eqref{c2}  and \eqref{c18}, we obtain
	{\begin{align*}
	&\frac{1}{\vartheta}	J(t_nu_n)+o_n(1)\n&=\frac{1}{\vartheta}\Big[J(t_nu_n)-\frac{1}{\theta p^+}\langle J'(t_nu_n), t_nu_n\rangle\Big]\n
	&=\frac{1}{\vartheta}\Big[M(\sigma(t_nu_n))-\frac{1}{\theta p^+}m(\sigma(t_nu_n))\rho_E(t_nu_n)\Big]+\frac{1}{2\theta p^+}
	{\int_{\Om}\Big(\int_{\Om}\frac{F(y,t_nu_n)}{|x-y|^{\mu(x,y)}}dy\Big)\frac{\mathcal{F}(x,t_nu_n)}{\vartheta}dx}\n
	&\leq\Big[M(\sigma(t_nu_n))-\frac{1}{\theta p^+}m(\sigma(t_nu_n))\rho_E(t_nu_n)\Big]+\frac{1}{2\theta p^+}
	{\int_{\Om}\int_{\Om}\frac{F(y,u_n(y))}{|x-y|^{\mu(x,y)}}\mathcal{F}(x,u_n(x))dxdy}\n
	&=\Big[M(1)\Big(\int_{\RR^N}\int_{\RR^N}\frac{1}{p(x,y)}\frac{|t_nu_n(x)-t_nu_n(y)|^{p(x,y)}}{|x-y|^{N+sp(x,y)}}dxdy\Big)^{\theta-1}\n &~~~~~~~~~~\times\int_{\RR^N}\int_{\RR^N}\Big(\frac{1}{p(x,y)}-\frac{1}{p^+}\Big)\frac{|t_nu_n(x)-t_nu_n(y)|^{p(x,y)}}{|x-y|^{N+sp(x,y)}}dxdy \Big]-\Psi(u_n)+\frac{1}{\theta p^+} I(u_n)\n
	&\leq\Big[M(1)\Big(\int_{\RR^N}\int_{\RR^N}\frac{1}{p(x,y)}\frac{|u_n(x)-u_n(y)|^{p(x,y)}}{|x-y|^{N+sp(x,y)}}dxdy\Big)^{\theta-1}\n &~~~~~~~~~~\times\int_{\RR^N}\int_{\RR^N}\left(\frac{1}{p(x,y)}-\frac{1}{p^+}\right)\frac{|u_n(x)-u_n(y)|^{p(x,y)}}{|x-y|^{N+sp(x,y)}}dxdy \Big]-\Psi(u_n)+\frac{1}{\theta p^+} I(u_n)\n
	&=J(u_n)-\frac{1}{\theta p^+}\langle J'(u_n), u_n\rangle=c+o_n(1), 
	\end{align*}}
	which contradicts \eqref{c17}. Hence $\{u_n\}$ is bounded in $E$. Therefore 
	passing to the limit $\gt$, if necessary, to a subsequence, thanks to {Theorem \ref{cpt}}, we have $u_n\rightharpoonup u$ weakly in $E, $ $u_n(x)\to u(x)$ a.e. in $\Om$ and  $u_n\to u$ strongly in $L^{\gamma(\cdot)}(\Om)$  for all $\gamma\in C_+(\ol\Om)$ with  $1<\gamma(x)<p_s^*(x).$ Let us denote $v_n:=u_n-u.$ Then clearly
	\begin{align}\label{c19}
	v_n\to 0 \text{ ~~~strongly in~} L^{\gamma(\cdot)}(\Om). 
	\end{align}
	In order to prove $\{u_n\}$  converges strongly to $u$ in $E$ as $\gt,$ we define the following functional. Let $\phi\in E$ be fixed and let $\mathcal{B}_\phi$ denote the linear functional on $E$ defined by
	{\begin{align*}\mathcal{B}_\phi(v)&=\int_{\RR^N}\int_{\RR^N}\frac{|
		\phi(x)-\phi(y)|^{p(x,y)-2}(\phi(x)-\phi(y))(v(x)-v(y))}{|
		x-y|^{N+s(x,y)p(x,y)}}dxdy+\int_{\Om}V|\phi|^{\p-2}\phi vdx\end{align*} }for all $v\in E.$
	For $(x,y)\in\R,$ let us denote {$$\Upsilon(x,y):=\frac{|\phi(x)-\phi(y)|}{|x-y|^{\frac{N}{p(x,y)}+s(x,y)}} \text {~~~and ~~} U(x,y):=\frac{|v(x)-v(y)|}{|x-y|^{\frac{N}{p(x,y)}+s(x,y)}}.$$} Then from H\"{o}lder inequality { and Lemma \ref{lemA1}, } it follows that
	{\begin{align}
		|\mathcal{B}_\phi(v)|&\leq C\Big[\||\Upsilon|^{p(\cdot,\cdot)-1}\big\|_{ {L}^{p'(\cdot,\cdot)}(\R)}
		\big\|U\big\|_{ L^{{p(\cdot,\cdot)}}(\R)}\n &~~~~~~~~~~~+
		\big\| |V^{\frac{1}{\overline{p}(\cdot)}}\phi|^{\overline{p}(\cdot)-1}\big\|_{ L^{\overline{p}'(\cdot)}(\Om)}\big\||V|^{\frac{1}{\overline{p}(\cdot)}} v\big\|_{ L^{\overline{p}(\cdot)}(\Om)}\Big]\n
		&\leq C\left[\left(\|\phi\|_{X_0}^{p^--1}+\|\phi\|_{X_0}^{p^+-1}\right)\|v\|_{X_0}+\left([\phi]_{V}^{p^--1}+[\phi]_{V}^{p^+-1}\right)[v]_{V}\right]\n&\leq C_1\left[\|\phi\|_E^{p^--1}+\|\phi\|_E^{p^+-1}\right]\|v\|_E\nonumber,
		\end{align}} where $C_1>0$ is a constant. 
	Thus for each $\phi\in E$ the linear functional $\mathcal{B}_\phi$ is continuous on $E$. Hence $u_n\rightharpoonup u$  weakly in $E$ implies that \begin{align}\label{c20}
	\DD\lim_{n\to+\infty}\mathcal{B}_{u}(v_n)=0.
	\end{align}
	Since $\{m(\sigma(u_n))-m(\sigma(u))\}$ is bounded a sequence in $\RR$, from \eqref{c20} we get
	\begin{align}\label{c21}
	\DD\lim_{n\to+\infty}	[m(\sigma(u_n))-m(\sigma(u))]\mathcal{B}_{u}(v_n)=0.
	\end{align}
	 Sine $\{u_n\}$ is bounded  in $E,$ using $(f1),$ Theorem \ref{cpt}    for each $n\in\mathbb N,$ we obtain $f(\cdot,u_n(\cdot))v_n(\cdot)\in {L^{q^+}(\Om)}$ and furthermore H\"older's inequality and \eqref{c19} infer
	{\begin{align}\label{c22}
		\|f(\cdot,u_n(\cdot))v_n(\cdot)\|_{ L^{q^+}(\Om)}
		&\leq 2\mathcal{C}\Big[\left(\int_{\Om} |v_n(x)|^{q^+}dx\right)^{\frac{1}{q^+}} +\left(\int_{\Om}|u_n(x)|^{(r(x)-1)q^+}~|v_n(x)|^{q^+}dx\right)^{\frac{1}{q^+}}\Big]\n
		&\leq  2 \mathcal{C}  \Big[\|v_n\|_{ L^{ q^+}(\Om)}+\|~|u_n|^{(r(\cdot)-1)q^+}\|_{ L^{\frac{r(\cdot)}{r(\cdot)-1}}(\Om)}^{\frac{1}{q^+}}\|v_n\|_{ L^{r(\cdot)q^+}(\Om)}\Big]\nonumber\\
		&\leq 2 \mathcal{C}\Big[\|v_n\|_{ L^{ q^+}(\Om)}
		+ \left(\|u_n\|^{(r^+-1)}_{ L^{r(\cdot)q^+}(\Om)}+\|u_n\|^{(r^--1)}_{ L^{r(\cdot)q^+}(\Om)}\right) 
		\|v_n\|_{ L^{r(\cdot)q^+}(\Om)}\Big]\nonumber\\
		&\leq C\Big[ \|v_n\|_{ L^{q^+}(\Om)}+\Big(\|u_n\|^{(r^+-1)}_{E}+\|u_n\|^{(r^--1)}_{E}\Big) 
		\|v_n\|_{ L^{r(\cdot)q^+}(\Om)}\Big]\n&=o_n(1).
		\end{align}}
	Similarly we can deduce that $f(\cdot,u_n(\cdot))v_n(\cdot)\in {L^{q^-}(\Om)}$ and
	\begin{align}\label{c23}
	\|f(\cdot,u_n(\cdot))v_n(\cdot)\|  _{ L^{q^-}(\Om)}=o_n(1).
	\end{align}
	Therefore using { Proposition \ref{HLS}, Theorem \ref{cpt}}, boundedness of $\{u_n\}$ in $E$ and $(f1)$  with  \eqref{c22} and \eqref{c23}, we get
	{\begin{align}\label{c24}
		&\int_{\Om}\int_{\Om}\frac{F(y,u_n(y))f(x,u_n(x))v_n(x)}{|x-y|^{\mu(x,y)}}dxdy\n
		&\leq C(N,q,\mu)\Big[\Big(\int_\Om|F(x,u_n(x))|^{q^+}dx\Big)^{1/{q^+}}\|f(\cdot,u_n(\cdot))v_n(\cdot)\|_{ L^{q^+}(\Om)}\nonumber\\
		&~~~~~~~~~~~~~~~~~+\Big(\int_\Om|F(x,u_n(x))|^{q^-}dx\Big)^{1/{q^-}}\|f(\cdot,u_n(\cdot))v_n(\cdot)\|_{ L^{q^-}(\Om)}\Big]\nonumber\\
		&\leq C\left(\|u_n\|_{ L^{ q^+}(\Om)}+\|~|u_n|^{r(\cdot)}~\|_{ L^{q^+}(\Om)}\right)\|f(\cdot,u_n(\cdot))v_n(\cdot)\|_{ L^{q^+}(\Om)}\nonumber\\
		&~~~~~~~+ C\left(\|u_n\|_{ L^{ q^-}(\Om)}+\|~|u_n|^{r(\cdot)}~\|_{ L^{q^-}(\Om)}\right)\|f(\cdot,u_n(\cdot))v_n(\cdot)\|_{ L^{q^-}(\Om)}\nonumber\\
		&\leq C\left[\|u_n\|_{ L^{ q^+}(\Om)}+\left\{\|u_n\|_{ L^{r(\cdot)q^+}(\Om)}^{r^+}+\|u_n\|_{ L^{r(\cdot)q^+}(\Om)}^{r^-}\right\}\right]\|f(\cdot,u_n(\cdot))v_n(\cdot)\|_{ L^{q^+}(\Om)}\nonumber\\
		&~~~~+ C\left[\|u_n\|_{ L^{ q^-}(\Om)}+\left\{\|u_n\|_{ L^{r(\cdot)q^-}(\Om)}^{r^+}+\|u_n\|_{ L^{r(\cdot)q^-}(\Om)}^{r^-}\right\}\right]\|f(\cdot,u_n(\cdot))v_n(\cdot)\|_{ L^{q^-}(\Om)}\nonumber\\
		&\leq C\left[\|u_n\|_{E}+\left\{\|u_n\|_{E}^{r^+}+\|u_n\|_{E}^{r^-}\right\}\right]\times\Big[\|f(\cdot,u_n(\cdot))v_n(\cdot)\|_{ L^{q^+}(\Om)}+\|f(\cdot,u_n(\cdot))v_n(\cdot)\|_{ L^{q^-}(\Om)}\Big]\n&=o_n(1).
		\end{align}}
	Now again from $(f1),$ it follows that $f(\cdot,u(\cdot))v_n(\cdot)\in L^{q^+}(\Om)\cap L^{q^-}(\Om)$  and hence by arguing similarly as above, we obtain 
	{\begin{align}\label{c25}
	& \int_{\Om}\int_{\Om}\frac{F(y,u(y))f(x,u(x))v_n(x)}{|x-y|^{\mu(x,y)}}dxdy\nonumber\\
	&\leq C\Big[\left\{\|u\|_{E}+\left(\|u\|_{E}^{r^+}+\|u\|_{E}^{r^-}\right)\right\}\Big\{\|f(\cdot,u(\cdot))v_n(\cdot)\|_{ L^{q^+}(\Om)}+\|f(\cdot,u(\cdot))v_n(\cdot)\|_{ L^{q^-}(\Om)}\Big\}\Big]\n&=o_n(1).
	\end{align}}
	Since $\{u_n\}$ is bounded,  combining \eqref{c19}, \eqref{c21},\eqref{c24} and \eqref{c25} we get
	{\begin{align}\label{c26}
		o_n(1)&=\langle J'(u_n)-J'(u), v_n\rangle\n
		&=m(\sigma(u_n))\mathcal{B}_{u_n}(v_n)-m(\sigma(u_n))\mathcal{B}_{u}(v_n)+\left[m(\sigma(u_n))-m(\sigma(u))\right]\mathcal{B}_{u}(v_n)\n&+\int_{\Om}\int_{\Om}\frac{F(y,u_n(y))f(x,u_n(x))v_n(x)}{|x-y|^{\mu(x,y)}}dxdy-\int_{\Om}\int_{\Om}\frac{F(y,u(y))f(x,u(x))v_n(x)}{|x-y|^{\mu(x,y)}}dxdy\n
		&=m(\sigma(u_n))[\mathcal{B}_{u_n}(v_n)-\mathcal{B}_{u}(v_n)]+o_n(1),
		\end{align}}
	that is, {\begin{align}\label{cnv}
	\lim_{n\to+\infty}\left[m(\sigma(u_n))\left(\mathcal{B}_{u_n}(v_n)-\mathcal{B}_{u}(v_n)\right)\right]=0.
	\end{align}}
	By using Remark \ref{main}, we have in particular
	{\begin{align}\label{c27}
	\lim_{n\to+\infty}[\mathcal{B}_{u_n}(v_n)-\mathcal{B}_{u}(v_n)]=0.
	\end{align}}
	Let us now recall well-known Simon inequalities,
	for all $\zeta,\xi\in \RR^N$, the followings hold:
	{ \begin{equation}\label{simon}
		\left\{ \begin{array}{rl}
		|\zeta-\xi|^p &\leq \frac{1}{p-1}\Big[ \left(|\zeta|^{p-2}\zeta-|\xi|^{p-2}\xi\right).\left(\zeta-\xi\right)\Big]^{\frac{p}{2}}
		\left(|\zeta|^p+|\xi|^p\right)^{\frac{2-p}{2}} ,1<p<2,\\
		|\zeta-\xi|^p	&\leq {2^p}\Big(|\zeta|^{p-2}\zeta-|\xi|^{p-2}\xi\Big).\left(\zeta-\xi\right)~,~~~~~~~~ p\geq 2.
		\end{array}
		\right.
		\end{equation}} Let us set 
	{\begin{align*} g_n^{(1)}(x,y)&:= \bigg[\frac{|
			u_n(x)-u_n(y)|^{p(x,y)-2}(u_n(x)-u_n(y))(v_n(x)-v_n(y))}{|
			x-y|^{N+s(x,y)p(x,y)}}\\
		&~~~~~~~~~~~~-\frac{|
			u(x)-u(y)|^{p(x,y)-2}(u(x)-u(y))(v_n(x)-v_n(y))}{|
			x-y|^{N+s(x,y)p(x,y)}}
		\bigg];\\
		 g_n^{(2)}(x,y)&:= \frac{|
			u_n(x)-u_n(y)|^{p(x,y)}} {|
			x-y|^{N+s(x,y)p(x,y)}};~
		g_n^{(3)}(x,y):=  \frac{|
			u(x)-u(y)|^{p(x,y)}} {|
			x-y|^{N+s(x,y)p(x,y)}};\\
		g_n^{(4)}(x)&:=V(x)\Big(|u_n|^{\p-2}u_n-|u|^{\p-2}u\Big) v_n(x).
		\end{align*}}
	For all $n\in\mathbb N,$ using  \eqref{simon}     we get $\DD\int_{\RR^N}\int_{\RR^N}g_n^{(1)}(x,y)dxdy\geq 0$ and $\DD\int_{\RR^N}\int_{\RR^N}g_n^{(4)}(x,y)dxdy$ $\geq 0$
	which together with \eqref{c27} and Remark \ref{main} infer that
	{\begin{align}\label{c27.1}
	\DD\lim_{n\to+\infty}\int_{\RR^N}\int_{\RR^N}g_n^{(1)}(x,y)dxdy=0
	\end{align}}
	and
	{\begin{align}\label{c27.2}
	\DD\lim_{n\to+\infty}\int_{\Om}g_n^{(4)}(x)dx= 0.
	\end{align}}
	In order to prove strong convergence of $\{u_n\}$ in $E$ by applying Simon's inequality, we divide our discussion into four cases. 
	Let us set
	{\begin{align*}
	&\Delta_1:=\{(x,y)\in \R:1<p(x,y)<2\} ;~~~~\Delta_2:=\{(x,y)\in \R:p(x,y)\geq2\};\n
	&\tilde{\Delta_1}:=\{x\in \Om:1<\p<2\} ;~~~~~~~~~~~~~~~~~~~~~~\tilde{\Delta_2}:=\{x\in \Om:\p\geq2\}.\n
	\end{align*}}
	
	\noi{$\rm(i)$ }Case $(x,y)\in\Delta_1.$ Sine $\{u_n\}$ is bounded in $E,$ using \eqref{simon}, \eqref{c27.1}, { generalized H\"{o}lder inequality and Lemma \ref{lemA1},}  we have
	{ \begin{align}\label{c28}
		&I_1:=\int_{\Delta_1}\frac{|
			v_n(x)-v_n(y)|^{p(x,y)}}{|
			x-y|^{N+s(x,y)p(x,y)}}dxdy\n
		&\leq \frac{1}{(p^{-}-1)}\int_{\Delta_1}\Big( g_n^{(1)}(x,y)\Big)^{\frac{p(x,y)}{2}}\Big( g_n^{(2)}(x,y)+g_n^{(3)}(x,y)\Big)^{\frac{2-p(x,y)}{2}}dxdy\n
		&\leq C \int_{\RR^N}\int_{\RR^N}\Big[\Big(( g_n^{(1)})^{\frac{p(x,y)}{2}}\cdot (g_n^{(2)})^{\frac{2-p(x,y)}{2}}\Big)+\Big( (g_n^{(1)})^{\frac{p(x,y)}{2}}\cdot (g_n^{(3)})^{\frac{2-p(x,y)}{2}}\Big)\Big] {dxdy}\nonumber\\
		&\leq C \Big[ \|( g_n^{(1)})^{\frac{p(\cdot,\cdot)}{2}}\|_{ L^{\frac{2}{p(\cdot,\cdot)}}(\RR^N\times\RR^N)}  
		\Big(\| (g_n^{(2)})^{\frac{2-p(\cdot,\cdot)}{2}}\|_{ L^{\frac{2}{2-p(\cdot,\cdot)}}(\RR^N\times\RR^N)}+\|(g_n^{(3)})^{\frac{2-p(\cdot,\cdot)}{2}}\|_{ L^{\frac{2}{2-p(\cdot,\cdot)}}(\RR^N\times\RR^N)}\Big)\Big]\nonumber \\
		&\leq C\Big[\Big( \| g_n^{(1)}\|_{ L^1(\RR^N\times\RR^N)}^{\frac{p^+}{2}} + \| g_n^{(1)}\|_{ L^1(\RR^N\times\RR^N)}^{\frac{p^-}{2}}\Big)\nonumber\\
		&\times \Big( \| g_n^{(2)}\|_{ L^1(\RR^N\times\RR^N)}^{\frac{2-p^+}{2}} + \| g_n^{(2)}\|_{ L^1(\RR^N\times\RR^N)}^{\frac{2-p^-}{2}}
		+\|g_n^{(3)}\|_{ L^1(\RR^N\times\RR^N)}^{\frac{2-p^+}{2}} + \| g_n^{(3)}\|_{ L^1(\RR^N\times\RR^N)}^{\frac{2-p^-}{2}} \Big) 
		\Big]=o_n(1).
		\end{align}}
	\noi {\rm(ii)} Case $(x,y)\in \tilde{\Delta_1}.$ Since $\{u_n\}$ is bounded, \eqref{simon} \eqref{c27.2}, {H\"{o}lder inequality and Lemma \ref{lemA1}} 
	imply
	{ \begin{align}\label{c29}
		&I_2:= \int_{\tilde{\Delta_1}}V(x)|v_n(x)|^{\p}dx\n&\leq \frac{1}{(p^{-}-1)}\int_{\tilde{\Delta_1}}(g_n^{(4)}(x))^{\p/2}\times\left\{V(x)
		(|u_n(x)|^{\p}+|u(x)|^{\p})\right\}^{\frac{2-\p}{2}}dx\n
		&\leq C\Big[\int_{\Om}(g_n^{(4)}(x))^{\frac{\p}{2}}(V(x)|u_n(x)|^{\p})^{\frac{2-\p}{2}}dx+\int_{\Om}(g_n^{(4)}(x))^{\frac{\p}{2}}(V(x)|u(x)|^{\p})^{\frac{2-\p}{2}}dx\Big]\n
		&\leq C\|(g_n^{(4)})^{\frac{\ol p(\cdot)}{2}}\|_{ L^{\frac{2}{\ol p(\cdot)}}(\Om)}
		\Big[~\|(V |u_n|^{\ol p(\cdot)})^{^{\frac{2-\ol p(\cdot)}{2}}}\|_{ L^{\frac{2}{2-\ol p(\cdot)}}(\Om)}
		+\|(V|u|^{\ol p(\cdot)})^{^{\frac{2-\ol p(\cdot)}{2}}}\|_{ L^{\frac{2}{2-\ol p(\cdot)}}(\Om)}\Big]\n
		&\leq C\Big[\|g_n^{(4)}\|_{ L^{1}(\Om)}^{\frac{ p^-}{2}}+\|g_n^{(4)}\|_{ L^{1}(\Om)}^{\frac{ p^+}{2}}\Big]\n &~~~~~~\times\Big[\|V |u_n|^{\ol p(\cdot)}\|_{ L^{1}(\Om)}^{\frac{2- p^-}{2}}+\|V|u_n|^{\ol p(\cdot)}\|_{ L^{1}(\Om)}^{\frac{2- p^+}{2}}+\|V |u|^{\ol p(\cdot)}\|_{ L^{1}(\Om)}^{\frac{2- p^-}{2}}+\|V|u|^{\ol p(\cdot)}\|_{ L^{1}(\Om)}^{\frac{2- p^+}{2}}\Big]
		=o_n(1).
\end{align} 
}
	{\rm(iii)} Case $(x,y)\in \Delta_2.$ Using \eqref{simon},\eqref{c27.1} and {generalized H\"{o}lder inequality}, we obtain
	{\begin{align}\label{c30}
		I_3:=\int_{\Delta_2}\frac{|v_n(x)-v_n(y)|^{p(x,y)}}{|x-y|^{N+s(x,y)p(x,y)}}dxdy\leq 2^{p^+}\int_{\RR^N}\int_{\RR^N}g_n^{(1)}(x,y)dxdy=o_n(1).
		\end{align}}
	{\rm (iv)} Case $(x,y)\in \tilde{\Delta_2}.$  Using \eqref{simon}, \eqref{c27.2} and {H\"{o}lder inequality}, we deduce
	{\begin{align}\label{c31}
	I_4:=\int_{\tilde{\Delta_2}}V(x)|v_n(x)|^{\p}dx\leq 2^{p^+}\int_{\Om}g_n^{(4)}(x)dx=o_n(1).
	\end{align}}
	Now taking into account \eqref{c27} and \eqref{c28}-\eqref{c31}, we get
	{ \begin{align*}
		{\rho}_E(v_n)&=\int_{\RR^N}\int_{\RR^N} \frac{|v_n(x)-v_n(y)|^{p(x,y)}} {|x-y|^{N+s(x,y)p(x,y)}}dxdy
		+\int_{\Om}V(x)|v_n(x)|^{\p} dx\n
		&=I_1+I_2+I_3+I_4=o_n(1).
		\end{align*}}
	Hence finally { Proposition \ref{norm-convergence}} yields that
	$$u_n\to u \text{ ~in ~} E \text{ ~ strongly~ as ~} \gt.$$ 
	{{\bf {Case}}}:$\DD\inf_{n\in \mathbb N }\|u_n\|_E= 0.$ If $0$ is an isolated point for the  sequence $\{\|u_n\|_E\}$ then there exists a subsequence $\{u_{n_k}\}$ of $\{u_n\}$ such that
	$$\DD\inf_{n\in \mathbb N }\|u_{n_k}\|_E= d_*>0$$ and therefore we can proceed as before. Otherwise, $0$ is an accumulation point of the sequence $\{\|u_{n_k}\|_E\}.$ Hence there
	is a subsequence $\{u_{n_k}\}$ of $\{u_n\}$ such that $u_{n_k}\to 0$ in $E$ strongly.\\
	{\noi Next we consider the non-degenerate case, i.e.,  $a>0.$ Hence the above proof reduces to Case $1$. Then making use of Remark \ref{rem-M} and  \eqref{m4} in place of Remark \ref{main} in the above and arguing in a similar fashion with some minute modifications,  the result follows.}
\end{proof}
\noi Now we state the following version of Mountain pass theorem.
{\begin{theorem}[Mountain pass theorem]\label{mp}
		Let $E$ be a real Banach space and suppose that $J\in C^1(E,\RR)$ satisfies the condition 
		$$\max\{J(0),J(u_1)\}\leq i<j\leq \inf_{\|u\|_E=\varrho_0}J(u)$$ for some $i<j, ~\varrho_0>0$ and $u_1\in E$ with $\|u_1\|_E=\varrho_0.$ Let $c\geq j$ be characterized by $$c_*=\DD\inf_{\nu\in\Gamma}\max_{0\leq t\leq1}J(\nu(t)),$$ where $\Gamma=\{\nu\in C\left([0,1],E\right), \nu(0)=0,\nu(1)=u_1\}$ is the set of continuous paths joining $0$ and $u_1.$ Then there exists a Cerami sequence $\{u_n\}\subset E$ such that 
		$$J(u_n)\to c\geq j \text{~ and~} (1+\|u_n\|_E)\|J'(u_n)\|_{E^*}\to 0 \text{ ~as~}\gt.$$
\end{theorem}}
\begin{proof}[Proof of Theorem \ref{mainthm}:] Since $J$ satisfies Lemma \ref{mp1} and Lemma \ref{mp2}, by Mountain Pass theorem {(Theorem \ref{mp}),} there exists a 
	Cerami sequence $\{u_n\}$ for $J$ in $E$ such that 
	\begin{align*}
	J(u_n)\to c_* \text{~~~~~ and ~~~~~} (1+\|u_n\|_E)\|J'(u_n)\|_{E^*}\to 0, \text { as } n\to\infty,
	\end{align*}
	where $c_*>0$ is the mountain pass level defined by $$c_*:=\DD\inf_{{\nu}\in\Gamma}\sup_{t\in[0,1]}J\left(\nu(t)\right).$$
	Now  Lemma \ref{bounded} implies $u_n\to u_*$ strongly in $E$ and hence $J'(u_*)=0.$ This infers that $u_*$is a critical point of $J$ and 
	therefore a weak solution to \ref{mainprob}. Also $J(u_*)=c_*>0$ and since
	$J(0)=0,$ finally we conclude that $u_*\not\equiv0.$ 
\end{proof}
\subsection{Proof of Theorem \ref{ground-state 2}}
\begin{proof}[Proof of Theorem \ref{ground-state 2}:]
	First we note that
	the condition $(f4)$ is the consequence of  $(f4)'.$ Indeed, for $t_2\geq t_1>0,$ $(f4)'$ implies
	{\begin{align*}
		&\mathcal{F}(x,t_2)-\mathcal{F}(x,t_1)\n&={\theta p^+}\left[\frac{2}{\theta p^+}\left(f(x,t_2)t_2-f(x,t_1)t_1\right)-\left(F(x,t_2)-F(x,t_1)\right)\right]\n
		&={\theta p^+}\Big[\int_0^{t_2} \frac{f(x,t_2)}{t_2^{\gr-1}} \tau ^{\gr-1}d\tau-\int_0^{t_1} \frac{f(x,t_1)}{t_1^{\gr-1}} \tau ^{\gr-1}d\tau-\int_{t_1}^{t_2} \frac{f(x,\tau)}{\tau^{\gr-1}} \tau ^{\gr-1}d\tau\Big]\n
		&={\theta p^+}\Big[\int_{t_1}^{t_2} \Big(\frac{f(x,t_2)}{t_2^{\gr-1}}-\frac{f(x,\tau)}{\tau^{\gr-1}}\Big) \tau ^{\gr-1}d\tau+\int_0^{t_1}\Big(\frac{f(x,t_2)}{t_2^{\gr-1}}- \frac{f(x,t_1)}{t_1^{\gr-1}}\Big) \tau ^{\gr-1}d\tau\Big]\geq 0.
		\end{align*}}
	Similarly for $0>t_1\geq t_2,$ we can deduce $\mathcal{F}(x,t_2)-\mathcal{F}(x,t_1)\geq 0,$ 
	that is $\mathcal{F}(\cdot,t)$ is increasing for $t\geq0$ and decreasing for $t\leq0.$ Hence $(f4)$ follows.
	Therefore there exists a weak  solution $v_*\not=0$ of \eqref{mainprob}, thanks to Theorem \ref{mainthm}, with
	$J'(v_*)=0$ and $J(v_*)=b_*,$ where $b_*$ is given as 
	$$b_*:=\DD\inf_{\nu\in\Gamma}\max_{0\leq t\leq1} J(\nu(t)),$$  where $\Gamma=\{u\in C([0,1],E):\nu(0)=0,J(\nu(1))<0\}$.
	We claim that $v_*$ is a ground state solution to \eqref{mainprob}.
	The Nehari manifold associated with the functional $J$ is defined as,
	$$\mathcal{N}:=\{u\in E\setminus\{0\}: \langle J'(u),u\rangle=0\}.$$ 
	Since $v_*$ is a critical point of $J,$ we have $v_*\in \mathcal{N}$. Let $$\al_*=\DD\inf_{u\in\mathscr{N}} J(u).$$
	Hence $\al_*\leq b_*.$ Therefore it is left to show  $b_*\leq \al_*.$
	For  $u\in\mathcal{N}$ define the fibering map $\mathcal{H}:[0,+\infty)\to\RR$ by $\mathcal{H}(t)=J(tu).$ Then $\mathcal{H}$ is differentiable with respect to $t$ and 
	{ \begin{align}\label{g1}
	\mathcal{H}'(t)=\langle J'(tu),u\rangle&=m(\sigma(tu))\Big[\int_{\RR^N}\int_{\RR^N}t^{p(x,y)-1}\frac{|u(x)-u(y)|^{p(x,y)}}{|x-y|^{N+s(x,y)p(x,y)}}dx dy\n&~~+
	\int_{\Om}t^{\p-1}V(x)|u|^{\p}dx\Big]
	-\int_{\Om}\int_{\Om}\frac{F(y,tu)f(x,tu)u(x)}{|x-y|^{\mu(x,y)}}dxdy.
	\end{align}}
	Also we have $\langle J'(u),u\rangle=0,$ that is, {\begin{align}\label{g2}
	m(\sigma(u))\rho_E(u)=I(u):=\int_{\Om}\int_{\Om}\frac{F(y,u)f(x,u)u(x)}{|x-y|^{\mu(x,y)}}dxdy.
	\end{align}}	
		 \underline {Case I}: $a=0.$ From $(M1)$, Remark \ref{main}, \eqref{g1} and \eqref{g2} for $t>1,$ it follows that 
	{\begin{align}\label{g2.0}
	\mathcal{H}'(t)&=m(\sigma(tu))\Big[\int_{\RR^N}\int_{\RR^N}t^{p(x,y)-1}\frac{|u(x)-u(y)|^{p(x,y)}}{|x-y|^{N+s(x,y)p(x,y)}}dx dy+
	\int_{\Om}t^{\p-1}V(x)|u|^{\p}dx\Big]\n&-\int_{\Om}\int_{\Om}\frac{F(y,tu)f(x,tu)u(x)}{|x-y|^{\mu(x,y)}}dxdy
	-t^{\theta p^+-1}m(\sigma(u))\rho_E(u)+t^{\theta p^+-1}I(u)\n
	&\leq b[\sigma(tu)]^{\theta-1}t^{p^+-1}\rho_E(u)-t^{\theta p^+-1}b[\sigma(u)]^{\theta-1}\rho_E(u)+\mathcal{G}(t)\n
	&\leq b t^{p^+(\theta-1)}[\sigma(u)]^{\theta-1}t^{p^+-1}\rho_E(u)-t^{\theta p^+-1}b[\sigma(u)]^{\theta-1}\rho_E(u)+\mathcal{G}(t)=\mathcal{G}(t),
	\end{align}}
	where {$$\mathcal{G}(t)=t^{\theta p^+-1}\int_{\Om}\int_{\Om}\frac{F(y,u)f(x,u)u(x)}{|x-y|^{\mu(x,y)}}dxdy-\int_{\Om}\int_{\Om}\frac{F(y,tu)f(x,tu)u(x)}{|x-y|^{\mu(x,y)}}dxdy.$$ }
	Since $\mathcal{F}(x,\tau)=2\tau f(x,\tau)-\theta p^+ F(x,\tau)\geq 0$ for all $x\in\RR^N, \tau\in\RR,$ it follows that {$$\frac{d}{dt}\frac{F(x,tu)}{t^{\gr}}=\frac{t^\gr f(x,tu)u(x)-\gr t^{\gr-1}F(x,tu)}{t^{\theta p^+}}=\frac{ f(x,tu)tu(x)-\gr F(x,tu)}{t^{\gr-1}}\geq 0.$$}
	Thus  we obtain 
	{\begin{align}\label{g3}
	\frac{F(x,tu)}{t^{\gr}} \text{~~ is increasing function in  $t> 0$ for all $u\in E$}.
	\end{align}}
	Now for $t>1,$ using $(f4)',$ \eqref{g3} and Remark \ref{f5}, we deduce  
{	\begin{align}\label{g4}
	\mathcal{G}(t)&= t^{\theta p^+-1}\Big[\int_{\Om}\int_{\Om}\frac{F(y,u)f(x,u)|u|^\gr}{|x-y|^{\mu(x,y)} |u|^{\gr-2}u}dxdy-\int_{\Om}\int_{\Om}\frac{F(y,tu)f(x,tu)|u|^{\gr}}{|x-y|^{\mu(x,y)}t^{\gr}|tu|^{\gr-2}tu}dxdy\Big]\n
	&\leq t^{\theta p^+-1}\Big[\int_{\Om}\int_{\Om}\frac{F(y,u)f(x,u)|u|^\gr}{|x-y|^{\mu(x,y)} |u|^{\gr-2}u}dxdy-\int_{\Om}\int_{\Om}\frac{F(y,tu)f(x,u)|u|^{\gr}}{|x-y|^{\mu(x,y)}t^{\gr}|u|^{\gr-2}~u}dxdy\Big]\n
	&= t^{\theta p^+-1}\Big[\int_{\Om}\Big(\int_{\Om}\frac{F(y,u)-\frac{F(y,tu)}{t^\gr}}{|x-y|^{\mu(x,y)}}dy\bigg){f(x,u)}u~ dx\Big]\leq 0.
	\end{align}}
	Combining \eqref{g2.0} and \eqref{g4}, we get
	$\mathcal{H}'(t)\leq 0 $ for $t>1.$ Arguing similarly as above we can deduce $\mathcal{H}'(t)\geq 0$ for $t\leq1.$ Therefore $1$ is the maximum point for $\mathcal{H},$ that is $J(u)=\DD\max_{t\geq 0}J(tu).$ Next we define the map $\nu:[0,1]\to E$ as $\nu(t)=(t_0u)t,$ where $t_0>1$ satisfies $J(t_0u)<0.$ This map is well-defined due to Lemma \ref{mp2}. So ${\nu}\in\Gamma.$ Hence 
	$$b_*\leq \DD \max_{0\leq t\leq1} J(\nu(t))\leq \DD \max_{0\leq t\leq1} J(tu)=J(u).$$ Since $u\in \mathcal{N}$ is arbitrary, we get $b_*\leq\al_*.$
	Therefore $$\DD\inf_{u\in\mathcal{N}}J(u)=\al_*=b_*=J(v_*).$$ 
	{ \underline {Case II}: $a>0.$ By replacing Remark \ref{main} with Remark \ref{rem-M} in \eqref{g2.0} and arguing in a similar way as in Case I  we can conclude the proof.} 
\end{proof}
\subsection{Proof of Theorem \ref{infinite-sol}}
To prove the Theorem \ref{infinite-sol} we need the Fountain theorem of Bartsch \cite[Theorem 2.5]{bartsch}; (see also \cite[Theorem 3.6]{willem}). 
The next lemma is due to \cite{fabian}.
\begin{lemma}\label{ftlem}
	Let $X$ be a reflexive and separable Banach space. Then there are $\{e_n\}\subset X$ and $\{f_n^*\}\subset X^*$ such that{$$X=\overline{span\{e_n:n=1,2,3..\}}, ~~X^*=\overline{span\{f_n^*:n=1,2,3..\}},$$} and { \begin{equation*}
		\langle f_i^*,e_j\rangle=
		\left\{ \begin{array}{rl}
		& 1  \text{~~~~if~~~} i=j\\
		& 0 \text{~~~~if~~~} i\not=j.
		\end{array}
		\right.
		\end{equation*}}
\end{lemma}
Let us denote {\begin{align}\label{not} 
X_n=span\{e_n\},~~~ Y_k=\bigoplus_{n=1}^k X_n \text{ ~~~and } Z_k=\overline{\bigoplus_{n=k}^\infty X_n}.\end{align} }
Now we recall the following Fountain theorem from  \cite{alves}.
\begin{theorem}[Fountain theorem]\label{ft}
	Assume that $J\in C^1(E,\RR)$ satisfies the Cerami condition $(C)_c$ and $J(-u)=J(u).$ If for each sufficiently large $k\in\mathbb N,$ there exists $\varrho_k>\delta_k>0$ such that
	\begin{itemize}
		\item[${\rm(B1)}$] $b_k:=\DD\inf\{J(u):u\in Z_k,~\|u\|_E=\delta_k\}\to+\infty,$ as $k\to+\infty,$
		\item[$\rm(B2)$] $a_k:=\max\{J(u):u\in Y_k,~\|u\|_E=\varrho_k\}\leq0.$
	\end{itemize}
	then $J$ has a sequence of critical points $\{u_k\}$ such that $J(u_k)\to+\infty.$
\end{theorem}

\begin{proof}[Proof of Theorem \ref{infinite-sol}:] For the reflexive, separable Banach space $E,$ define $Y_k$ and $Z_k$ as in \eqref{not}. Now $J$ satisfies Cerami condition $(C)_c$  for all $c\in \RR,$ thanks to Lemma \ref{bounded}, and $J$ is even. So to prove the Theorem \ref{infinite-sol}, it remains to verify conditions $\rm(B1)-\rm(B2).$\\ 
	\textit{Verification of $\rm(B1):$}
	For $k$ large enough let us denote
	\begin{align}\label{al}
	\al_k=\sup_{u\in Z_k,~\|u\|_{E}=1}\|u\|_{L^{\gamma(\cdot)}(\Om)},\end{align} where ${\gamma}\in C_+(\ol\Om)$ such that $1< \gamma(x)<p_s^*(x)$ for all $x\in\ol\Om.$ Then we have
	\begin{align}\label{ft1}
	\DD\lim_{k\to+\infty}\al_k= 0.
	\end{align} If not, supposing to the contrary, there exist $\epsilon_0>0,k_0\geq0$ and a sequence $\{u_k\}$ in $Z_k$ such that $$\|u_k\|_E=1\text{ and}~~
	\|u_k\|_{L^{\gamma(\cdot)}(\Om)}\geq \epsilon_0 $$
	for all $k\geq k_0.$
	Since $\{u_k\}$ is bounded in $E,$ there exists $u_0\in E$ such that up to a subsequence, still denoted by $\{u_k\},$ we have $u_k\rightharpoonup u_0$ weakly in $E$ as $k\to+\infty$ and
	$$\langle f_j^*,u_0\rangle=\DD\lim_{k\to+\infty}\langle f_j^*,u_k\rangle=0$$  for $j=1,2,3,\cdots.$ Thus we have $u=0.$ Furthermore, using { Theorem \ref{cpt}} we obtain
	$$\epsilon_0\leq\DD\lim_{k\to+\infty}\|u_k\|_{L^{\gamma(\cdot)}(\Om)}=\|u_0\|_{L^{\gamma(\cdot)}(\Om)}=0,$$ which is a contradiction to the fact $\epsilon_0>0.$ Hence \eqref{ft1} holds true. Let $u\in Z_k$ with $\|u\|_E>1.$ By  using Remark \ref{main} {(or Remark \ref{rem-M})} and { Proposition \ref{norm-modular}}, we have
	{\begin{align}\label{ft2}
	J(u)
	\geq \frac{M(1)}{(p^+)^{\theta}}\{\rho(u)\}^{\theta}-\Psi(u)
	\geq\frac{M(1)}{(p^+)^{\theta}}\|u\|_E^{  \theta p^-}-\Psi(u).
	\end{align}} 
	Now \eqref{ft1} infers that $\al_k<1$ for large $k.$ Therefore using $(f1),$ Lemma \ref{lemA1}, {Proposition \ref{HLS}} and  \eqref{al} for sufficiently large $k,$  we get 
	{\begin{align}\label{ft3}
		\Psi(u)&
		\leq C\Big[\Big\{\|u\|_{ L^{q^-}(\Om)}^2+\Big(\|u\|_{ L^{r(\cdot)q^-}(\Om)}^{2r^-}+\|u\|_{ L^{r(\cdot)q^-}(\Om)}^{2r^+}\Big)\Big\}\n&~~~~~~~~~~~~~~~+\Big\{\|u\|_{ L^{q^+}(\Om)}^2+\Big(\|u\|_{ L^{r(\cdot)q^+}(\Om)}^{2r^-}+\|u\|_{ L^{r(\cdot)q^+}(\Om)}^{2r^+}\Big)\Big\}\Big]\n
		&\leq 2C\left\{\|u\|_{E}^2 \al_k^2+\left(\|u\|_{E}^{2r^-}\al_k^{2r^-}+\|u\|_{E}^{2r^+}\al_k^{2r^+}\right)\right\}\leq {\tilde{ C }}\al_k \|u\|_{E}^{2r^+},
		\end{align}}
	where ${\tilde{ C }}$ is a positive constant.
	Thus \eqref{ft2} and \eqref{ft3} give us 
	{\begin{align}\label{ft4}
		J(u)\geq \frac{M(1)}{ (p^+)^\theta}\|u\|_E^{\theta  p^-}-{\tilde{ C }} \al_k \|u\|_{E}^{2r^+}.
		\end{align}}
	Consider the real function $G:\RR\to \RR,$ 
	{$$G(t)= \frac{M(1)}{ (p^+)^\theta}t^{ \theta p^-}-{\tilde{ C }} \al_k t^{2r^+}.$$ }Then from elementary calculus it follows that $G$ attains its maximum at {$$\delta_k={\left(\frac{M(1)\theta p^-}{2r^+(p^+)^\theta{\tilde{ C }}\al_k  }\right)^{1/(2r^+-\theta p^-)}.}$$} Therefore the maximum value of $G$ is given by 
	{\begin{align}
		G(\delta_k)&=\frac{M(1)}{ (p^+)^\theta}\left(\frac{M(1)\theta p^-}{2r^+(p^+)^\theta{\tilde{ C }}\al_k  }\right)^{\theta p^-/(2r^+-\theta p^-)}-{\tilde{ C }}\al_k\left(\frac{M(1)\theta p^-}{2r^+(p^+)^\theta{\tilde{ C }}\al_k  }\right)^{2r^+/(2r^+-\theta p^-)}\n
		&=\left(\frac{M(1)}{ (p^+)^\theta}\right)^{2r^+/(2r^+-\theta p^-)}\left(\frac{1}{{\tilde{ C }} \al_k}\right)^{\theta p^-/(2r^+-\theta p^-)}\left(\frac{\theta p^-}{2r^+}\right)^{\theta p^-/(2r^+-\theta p^-)}\bigg(1-\frac{\theta p^-}{2r^+}\bigg)\nonumber.
		\end{align}}
	Since $\theta p^-<2r^+$ and $\al_k \to 0$ as $k \to +\infty,$ we have 
	{\begin{align}\label{ft5}
	G(\delta_k)\to +\infty \text{~~as~} k\to +\infty.
	\end{align}}
	Again  using \eqref{ft1}, we get $\delta_k\to+\infty$ as $k\to +\infty.$ Thus for $u\in Z_k$ with $\|u\|_E=\delta_k,$  combining \eqref{ft4} and \eqref{ft5}, it readily follows that as $k\to +\infty$
	$$b_k=\DD\inf_{u\in Z_k,\|u\|_E=\delta_k} J(u)\to +\infty.$$
	\textit{Verification of $\rm(B2):$} Due to the presence of Choquard type nonlinearity here we will use an indirect argument. Suppose assertion $(B2)$ of Fountain theorem does not hold true for some given $k.$ Then there exists a sequence $\{u_n\}\subset Y_k$ such that 
	\begin{align}\label{ft6}
	\|u_n\|_E\to +\infty, ~~~~~J(u_n)\geq 0.
	\end{align} 
	Let us take $w_n:=\frac{u_n}{\|u_n\|_E},$ then $w_n\in E$ and $\|w_n\|_E=1.$ Since $Y_k$ is of finite dimension, there exists $w \in Y_k\setminus\{0\}$ such that up to a subsequence, still denoted by $\{w_n\},$ $w_n\to w$  strongly and $w_n(x)\to w(x)$
	a.e. $x\in\RR^N$ as $n \to +\infty. $
	If $w(x)\not=0$ then $|u_n(x)|\to+\infty$ as $n \to +\infty.$
	{Similar to \eqref{c9}, it follows that for each $x\in\Om$
		\begin{align}\label{ft6.0}
		\bigg(\int_{\Om}\frac{F(y,u_n(y))|w_n(y)|^{\gr}}{|x-y|^{\mu(x,y)}|u_n(y)|^{\gr}}dy\bigg
		)\frac{F(x,u_n(x))}{|u_n(x)|^{\gr}}|w_n(x)|^{\gr}\to+\infty .
		\end{align}
		Thus using \eqref{ft6}, \eqref{ft6.0} with Remark \ref{f5} and applying Fatou's lemma, as $\gt$ 
		{\begin{align}\label{ft7}
		\frac{\Psi(u_n)}{\|u_n\|_E^{\theta p^+}}&=\frac{1}{2}\int_{\Om}\Big(\int_{\Om}\frac{F(y,u_n)|u_n(y)|^{\gr}}{|x-y|^{\mu(x,y)}|u_n(y)|^{\gr}}dy\Big
		)\frac{F(x,u_n)}{|u_n(x)|^{\gr}}|w_n(x)|^{\gr}dx\to+\infty.
		\end{align}}}
	Since  $\|u_n\|_E>1$ for large $n$, using  Remark \ref{main} {(or Remark \ref{rem-M})},,  Proposition \ref{norm-modular} and \eqref{ft7}, we deduce as $n\to+\infty$
{	\begin{align*}
	J(u_n)
	\leq  \frac{M(1)}{(p^-)^\theta}{\|u_n\|}^{\theta p^+}-\Psi(u_n)=\bigg(\frac{M(1)}{(p^-)^\theta}-\frac{1}{\|u_n\|_E^{\theta p^+}}\Psi(u_n)\bigg)\|u_n\|_E^{\theta p^+}\to-\infty,
	\end{align*}}
	\noi which contradicts \eqref{ft6}. Thus for sufficiently large $k$  we can have $\varrho_k>\delta_k>0$  such that for $u\in Y_k$ with $\|u\|_E=\varrho_k$ the assertion  $\rm(B2)$ follows.\\
\end{proof}
\subsection{Proof of Theorem \ref{dual-fount-sol}}To prove the Theorem \ref{infinite-sol} we need the  Bartsch-Willem Dual fountain theorem (see \cite[Theorem 3.18]{willem} ).
Since $E$ is reflexive separable Banach space, using Lemma \ref{ftlem} we can define $Y_k$ and $Z_k$ appropriately.
\begin{definition}  We say that $J$ satisfies the $(C)_{c}^{*}$ condition (with respect to $Y_{k}$) if any sequence $\{u_k\}$ in $E$ with $u_{k}\in Y_{k}$ such that
	$$J(u_{k})\to c \text {~~and~~} \|J'_{|_{Y_{k}}}(u_{k})\|_{E^*}(1+\|u_{k}\|_{E})\to 0, \text{~~as~~~} k\to+\infty$$ contains a subsequence converging to a critical point of $J,$ where $E^*$ is the dual of $E.$
\end{definition} 
\begin{theorem}[Dual fountain Theorem]\label{dual}
	Let  $J\in C^1(E,\RR)$  satisfy $J(-u)=J(u).$ If for each $k\geq k_0$ there exist $\varrho_{k}>\delta_{k}>0$ such that
	\begin{itemize}
		\item[$(A_{1})$] $a_{k}=\inf\{J(u):u\in Z_{k},\,\,\|u\|_{E}=\varrho_{k}\}\geq 0;$
		\item[$(A_{2})$] $b_{k}=\sup\{J(u):u\in Y_{k},\,\,\|u\|_{E}=\delta_{k}\}< 0;$
		\item[$(A_{3})$] $d_{k}=\inf\{J(u):u\in Z_{k},\,\,\|u\|_{E}\leq \varrho_{k}\}\to 0$ as $k\to+\infty;$
		\item[$(A_{4})$] $J$ satisfies the $(C)_{c}^*$ condition for every $c\in [d_{k_0},0[.$
	\end{itemize}
	Then $J$ has a sequence of negative critical values converging to $0$.
\end{theorem}
\begin{remark} Here we would like to mention that  in { \cite{willem}}, assuming that $J$ satisfies $(PS)_c^*$  condition the Dual fountain theorem is proved using Deformation theorem which is still valid under Cerami condition. Therefore we see that like many critical point theorems the Dual fountain theorem is true under $(C)_c^* $ condition.
\end{remark}
\begin{lemma}\label{Cc'} 
	Suppose that the hypotheses in Theorem \ref{dual-fount-sol} hold, then $J$ satisfies the $(C)_{c}^{*}$ condition.
\end{lemma}
\begin{proof}
	Let $c\in \mathbb{R}$ and the sequence $\{u_k\}$ in $ E$ be such that  $u_{k}\in Y_{k}$ for all $k\in \mathbb{N},$ $J(u_{k})\to c$ and
	$\big\|J'{_{|_{Y_{k}}}}(u_{k})\big\|_{E^*}(1+\|u_{k}\|_{E})  \to 0,$ as 
	$k\to +\infty.$
	Therefore, we have
{	$$c=J(u_{k})+o_{k}(1) \mbox{ and }\langle J'(u_{k}),u_{k} \rangle=o_{k}(1).$$}
	Analogously to the proof of  Lemma \ref{bounded}, we can show that
	$\{u_k\}$ is bounded in $E$. Hence there exists a subsequence, still denoted by  $\{u_k\},$ and $u \in E $ such that $u_{k}\rightharpoonup u$ weakly  in $E$ as $k \to +\infty$. On the other hand, Lemma \ref{ft} implies $E=\overline{\cup_{k}Y_{k}}=\overline{span\{e_{k}:k\geqslant 1\}}$ and thus we can choose $v_{k}\in Y_{k}$ such that $v_{k}\to u$ strongly  in $E$ as $k \to +\infty$.
	Therefore, using the facts $J'{_{|_{Y_{k}}}}(u_k)\to 0$ and  $u_k-v_k\rightharpoonup 0$ in $Y_{k},$ (see {\cite[Proposition 3.5]{brezis}}),  we achieve 
{	\begin{align*}
	\lim_{k\to\infty}	\langle J'(u_k),u_k-u\rangle=\lim_{k\to\infty}\langle J'(u_k),u_k-v_k\rangle+ \lim_{k\to\infty}\langle J'(u_k),v_k-u\rangle=0.
	\end{align*}}Again recalling the proof of Lemma \ref{bounded}, we can deduce $u_k\to u$ strongly in $E$  as $k\to\infty.$	
	Then, we conclude that $J$ satisfies the $(C)_{c}^{*}$ condition. Thus, we obtain that $J'(u_k)\to  J'(u)$ as $k\to +\infty.$ Let us prove $J'(u)=0$. Indeed, taking $\omega_{j}\in Y_{l}$,  for $k\geq j,$ we have
{	\begin{align*}
		\langle J'(u),\omega_{j}\rangle&=	\lim_{k\to+\infty}[\langle J'(u)-J'(u_k),\omega_{j}\rangle+\langle J'(u_k),\omega_{j}\rangle]\n
		&=	\lim_{k\to+\infty}[\langle J'(u)-J'(u_k),\omega_{j}\rangle+\langle J'{_{|_{Y_{k}}}}(u_k),\omega_{j} \rangle]=0.
	\end{align*}}
	Therefore, $J'(u)=0$ in $E^*$  and hence $J$ satisfies the $(C)_{c}^{*}$ condition for every $c\in \mathbb{R}.$
\end{proof}
\begin{proof}[Proof of Theorem \ref{dual-fount-sol}]
	For the reflexive, separable Banach space $E,$ define $Y_k$ and $Z_k$ as in \eqref{not}. $J$ is even and Lemma \ref{Cc'} ensures that $J$ satisfies Cerami condition $(C)_c^*$  for all $c\in \RR.$ So to prove Theorem \ref{dual-fount-sol} it is enough to verify conditions $\rm(A1)-\rm(A3).$\\
	\textit{Verification of $\rm(A1)$:} For all $u\in Z_k$ with $\|u\|_E<1,$ arguing in a similar fashion as \eqref{ft3} we can derive
	{\begin{align}\label{df0}
	\Psi(u)\leq C_4\al_k\|u\|_E
	\end{align}} which together with Remark \ref{main} {(or Remark \ref{rem-M})} and Proposition \ref{norm-modular} imply
	{\begin{align}\label{df1}
		J(u)\geq \frac{M(1)}{(p^+)^\theta}\|u\|^{\theta p^+}_E-C_4\al_k\|u\|_E,
		\end{align}} where $C_4>0$ is some constant. Let us choose $\varrho_k=[(p^+)^\theta C_4 \al_k/{M(1)}]^{1/{(\theta p^+-1)}}.$ Since $\theta p^+>1,$  \eqref{al} yields that \begin{align}\label{df2}
	\varrho_k\to0 \text {~~as~} k\to+\infty.\end{align}
	Thus for $u\in Z_k$ with $\|u\|_E=\varrho_k$ and for sufficiently large $k,$ from \eqref{df1} we get $J(u)\geq0.$ \\
	\textit{Verification of $\rm(A2)$:}  
	{Suppose assertion $(A2)$ of Dual fountain theorem does not hold true for some given $k.$ Then there exists a sequence $\{u_n\}\subset Y_k$ such that 
		\begin{align}\label{df6}
		\|u_n\|_E\to +\infty, ~~~~~J(u_n)\geq 0.
		\end{align} }
	Now arguing similarly as in the proof of assertion $(B2)$ of Theorem \ref{ft}, we get \eqref{ft6.0} and \eqref{ft7} which together with Remark \ref{main}{(or Remark \ref{rem-M})} and  Proposition \ref{norm-modular}  imply that as $\gt$
	{\begin{align*}
	J(u_n)
	&\leq  \frac{M(1)}{(p^-)^\theta}{\|u_n\|}^{\theta p^+}-\Psi(u_n)
	=\bigg(\frac{M(1)}{(p^-)^\theta}-\frac{1}{\|u_n\|_E^{\theta p^+}}\Psi(u_n)\bigg)\|u_n\|_E^{\theta p^+}\to-\infty.
	\end{align*}}
	Hence we get a contradiction to \eqref{df6}.
	Thus there exists $k_0\in\mathbb N$ such that for all $k\geq k_0$  we have $1>\varrho_k>\delta_k>0$  such that for $u\in Y_k$ with $\|u\|_E=\delta_k$ the assertion  $\rm(A2)$ follows.\\
	\textit{Verification of $\rm(A3)$:} Since $Y_k \cap Z_k\not=\emptyset,$ we get 
	$d_k\leq b_k<0.$ Now for $u\in Z_k,$ $\|u\|_E\leq \varrho_k$ using \eqref{df0}, we have
	{\begin{align*}
	J(u)\geq -C_4\al_k\|u\|_E\geq -C_4\al_k\varrho_k.
	\end{align*} }Therefore using \eqref{al} and \eqref{df2}, we obtain 	
	{\begin{align*}
	d_k\geq -C_4\al_k\varrho_k\to 0 \text{~~ as~} k\to\infty.
	\end{align*}} Since $d_k<0,$ we  finally conclude $\DD\lim_{k\to\infty} d_k=0.$ 
	Thus the proof of the theorem is complete.
\end{proof}
\section{Acknowledgement}
\noi Sweta Tiwari is supported by MATRICS grant MTR/2018/000010 funded by 
SERB, India and Reshmi Biswas is supported by research
grant from IIT Guwahati. 
The authors wish to acknowledge Prof. V. D. R\u{a}dulescu for useful comments and valuable suggestions which have helped to improve the presentation.

\end{document}